\documentclass[final,hidelinks,onefignum,onetabnum]{siamart251216}

\usepackage{amsfonts}
\usepackage{graphicx}           
\usepackage{float}
\usepackage{subcaption}
\usepackage{setspace}
\usepackage[english]{babel}

\usepackage{mathrsfs}
\usepackage{booktabs}

\newtheorem{thm}{Theorem}[section]

\newtheorem{lem}[thm]{Lemma}

\newtheorem{remark}{Remark}

\headers{Level-set PINNs for domain inverse problems of gravimetry}{Yao, Li, and Qian}
\title{Level-set physics-informed neural networks for domain inverse problems of gravimetry}

\author{
  Jingnan Yao\thanks{
                   School of Science, Harbin Institute of Technology, Shenzhen,
                   Shenzhen, 518055, China
                   (\email{24B958002@stu.hit.edu.cn}, \email{liwenbin@hit.edu.cn}).}
  \and
  Wenbin Li\footnotemark[1] \thanks{Corresponding author.}
  \and
  Jianliang Qian\thanks{Department of Mathematics and Department of CMSE,
                       Michigan State University, East Lansing, MI 48824, USA
                       (\email{jqian@msu.edu}).}
}

\begin{document}
\maketitle

\begin{abstract}
We propose level-set physics-informed neural networks (PINNs) for domain inverse problems of gravimetry. The domain inverse problem establishes a correctness class for ill-posed inverse gravimetry, which we solve within the PINNs framework. Directly representing the domain inverse problem via neural networks is problematic due to the discontinuous nature of interfaces. We consider a level-set formulation where the neural network represents a continuous level-set function, and its zero level-set depicts sharp interfaces. To overcome the challenges of exploding and vanishing gradients caused by sharp interfaces during training, we propose an interface-aware backpropagation strategy. By redefining the derivative associated with interface evolution, this strategy enables a broader support region to drive the evolution process. Detailed analysis is provided to justify the efficacy of this strategy. Additionally, we introduce a simple procedure for adaptive refinement of collocation points near interfaces. The selection of network architecture is investigated by studying the solution spaces and the approximation properties of neural networks. Finally, extensive 2D and 3D numerical examples demonstrate the effectiveness of the proposed method.
\end{abstract}



\section{Introduction}
Gravimetric surveys are valuable in the exploration of large-scale and deep-seated structures. The primary objective of gravity inversion is to reconstruct the internal density distribution of the Earth using surface measurements of the gravity potential or related anomaly data \cite{zhd15}. It is also useful in deep-space exploration and planetary science, allowing scientists to probe the deep internal structures of celestial bodies \cite{wiephi98,zuberetc13}.
However, the inverse problem of gravimetry is notoriously ill-posed. As dictated by the equivalent source principle, an infinite number of subsurface mass configurations can produce identical gravity responses, leading to severe non-uniqueness in the Hadamard sense \cite{isa90}.

To tackle the ill-posedness of inverse gravimetry, classical methodologies have relied on constructing regularization terms based on prior geological or mathematical assumptions \cite{hansen92}. Prominent strategies include depth-weighting regularization, which is designed to counteract the natural decay of potential fields with depth \cite{liold98}, and sparsity-promoting regularization, which favors the recovery of compact anomalous bodies with sharp boundaries \cite{far08}. Driven by advancements in well-posedness theory, an alternative and highly effective paradigm has emerged: mitigating ill-posedness by explicitly restricting the mathematical form of the solution. A particularly successful framework parameterizes the anomalous mass as a volume distribution $ f\chi_D$, where $f$ denotes the density contrast and $\chi_D$ is the characteristic function of the domain $D$. Given a priori knowledge of $f$ and certain geometric constraints on $D$, the domain inverse problem of gravimetry admits a unique solution \cite{isa90}. This uniqueness result has lead to the use of level-set methods in inverse gravimetry \cite{isaleuqia11,isaleuqia13,liqia21,liqia22,liden26software}, where the unknown domain $D$ is recovered implicitly by tracking its boundary $\partial D$ as the zero-level-set of a level-set function.

Recently, the application of neural networks and deep learning strategies has emerged as a growing trend in the study of inverse gravimetry. End-to-end learning approaches train neural networks to directly map gravity data to density reconstructions \cite{hualiuqizha21,zhazhaliufan22,yanhuliujiewanche22,zhochelvwan23,cheli24}. Although efficient, these end-to-end methods exhibit inherent limitations regarding solution reliability and generalization \cite{cheli24}. Specifically, when a significant discrepancy exists in the mass distribution between the prediction and training datasets, or when the gravity data sampling scheme in the inference task deviates from that used during training, end-to-end neural networks struggle to yield correct results. To enhance reliability and generalization, a prevalent trend involves expanding the training dataset and employing diverse geological models \cite{hualiuqizha21,zhazhaliufan22,zhochelvwan23}. In contrast, \cite{cheli24} advocates restricting learning and prediction to the correctness class of the inverse gravimetry problem. However, none of these approaches fundamentally resolve the underlying issue: without the incorporation of physical models, end‑to‑end neural networks remain inherently limited.

Physics-Informed Neural Networks (PINNs) establish a new paradigm that incorporates governing physical laws, typically partial differential equations (PDEs), into the training of neural networks \cite{raiperkar19}. Both the data misfit and the residual of governing equations are penalized in the loss function. In this way, the network is trained not just to fit the empirical data, but also to obey the underlying physical models. The paradigm of PINNs has been widely adopted to solve a broad spectrum of forward and inverse problems; comprehensive reviews can be found in, e.g., \cite{cuodietc22,nganinetc23}. In the context of elliptic inverse source problems, for instance, \cite{daijinsauzho25} employs the framework of PINNs to solve elliptic optimal control problems, treating the unknown source as the control variable and embedding the Karush-Kuhn-Tucker (KKT) optimality system as penalty terms in the loss function. Moreover, \cite{hujinzho26} develops a singularity-enrichment neural network for identifying point sources in the Poisson equation from Cauchy data. This method decomposes the state variable into an analytical singular part, represented by the fundamental solution, and a regular part that is resolved using the paradigm of PINNs.

A characteristic of PINNs is their inherent restriction to modeling sufficiently smooth solutions. This restriction stems partly from the various regularization terms incorporated into the loss function to ensure training stability; more fundamentally, it is dictated by the intrinsic continuity of neural network architectures. For domain inverse problems involving discontinuous interface structures, applying the framework of PINNs therefore requires specialized strategies. There are some recent studies on PINNs for solving partial differential equations involving interfaces \cite{hulinlai22,tselinhulai23,tselai23,hushilinlai24}. In \cite{hulinlai22}, it solves the forward problem of an elliptic equation where the solution admits discontinuities on a given interface. The key innovation is to represent a $d$-dimensional piecewise continuous solution by a continuous augmented function defined in ($d+1$)-dimensional space. Then PINNs can be naturally adopted to solve the augmented continuous function. In \cite{tselinhulai23,tselai23}, the idea of continuous function extension, i.e. expressing a d-dimensional discontinuous solution by a (d+1)-dimensional continuous neural network function, is extended to capture cusps, where the solution itself is continuous but has a jump discontinuity in its normal derivative across a given interface. Although the discontinuity representation strategy in \cite{hulinlai22,tselinhulai23,tselai23} shares conceptual similarities with the level-set method, its application is restricted to forward problems. As a result, the rich interface-tracking machinery of the level-set method remains underutilized. Additionally, it requires lifting the $d$-dimensional discontinuous solution to $(d+1)$ dimensions. In \cite{hushilinlai24}, it solves the forward problem of partial differential equations on static and evolving surfaces using PINNs. The level-set function is employed to represent a given surface, or an evolving surface with a prescribed velocity field, so that the surface normal and mean curvature in the PDE solutions can be easily computed. However, the work does not address the inverse problem of surface reconstructions. More recently, \cite{lujuzhu25} proposes Multi-TransNet to solve elliptic equations across domains separated by interfaces. By integrating multiple distinct TransNets via a non-overlapping domain decomposition approach, the method isolates discontinuities at subdomain boundaries and enforces interface jump conditions within the physical loss function. Nevertheless, like the former works, it focuses on the forward problem rather than the inverse problem of interface reconstructions.

In this work, we propose level-set physics-informed neural networks (PINNs) for domain inverse problems of gravimetry. The shape of the unknown domain is represented by the zero level-set of a continuous neural network, without lifting the spatial coordinates into a higher dimension. We adopt the paradigm of PINNs to construct the physical loss function, and leverage the interface evolution mechanism inherent to the level-set method to drive the network training. The remainder of the paper is organized as follows. Section \ref{sec2} explains the domain inverse problems of gravimetry and their related well-posedness results. Section \ref{sec3} proposes the methodology of level-set physics-informed neural networks (PINNs). In this section, we construct the loss function of level-set PINNs, develop an interface-aware backpropagation algorithm for network training, and introduce a simple strategy for adaptive refinement of collocation points near interfaces; moreover, we study the choice of network architectures by investigating their solution spaces and approximation properties. Section \ref{sec:results} presents both 2D and 3D numerical examples to demonstrate the algorithm's efficacy. Section \ref{sec5} draws conclusions.

\section{Inverse gravimetry and domain inverse problems} \label{sec2}
The gravitational potential $U$ generated by an anomalous density $\rho$ with $\operatorname{supp}\rho\subset\Omega$ can be described by a Fredholm integral equation of the first kind,
\begin{equation}\label{eqn1}
U(\mathbf{y};\rho)=\gamma\int_\Omega K(\mathbf{y},\mathbf{x})\rho(\mathbf{x})\,\mathrm{d}\mathbf{x}\,.
\end{equation}
The integral kernel $K(\mathbf{y},\mathbf{x})=K(|\mathbf{y}-\mathbf{x}|)$ is the fundamental solution to Laplace's equation,
\begin{equation}\label{eqn2}
K(\mathbf{y},\mathbf{x})=\left\{\begin{array}{ccc}
-\frac{1}{2\pi}\mathrm{ln}|\mathbf{y}-\mathbf{x}| &,&\mathbf{x},\mathbf{y}\in\mathbf{R}^2\,,\\
\frac{1}{4\pi|\mathbf{y}-\mathbf{x}|} &,&\mathbf{x},\mathbf{y}\in\mathbf{R}^3\,.
\end{array}\right.
\end{equation}
The constant $\gamma$ is related to the universal gravitational constant $\gamma_0 = 6.67384 \times 10^{-8} \ \mathrm{cm}^3 \mathrm{g}^{-1} \mathrm{s}^{-2}$; in $\mathbf{R}^d$ ($d=2,3$), it is defined as $\gamma = 2^{d-1}\pi \gamma_0$. 
Alternatively, the gravitational potential can be modeled by the partial differential equation (PDE)
\begin{equation} \label{eqn3}
-\Delta U(\mathbf{x})=\gamma\rho(\mathbf{x})\,.
\end{equation}
For a compactly supported density $\rho(\mathbf{x})$ ($\operatorname{supp}\rho\subset\Omega$), the far-field boundary condition is,
\begin{equation} \label{eqn4}
\left\{\begin{array}{ccc}
U(\mathbf{x})\sim -\frac{\gamma M}{2\pi}\mathrm{ln}|\mathbf{x}| &,&\mathrm{as} \  |\mathbf{x}|\to\infty \quad  (\mathbf{x}\in\mathbf{R}^2)\,,\\
U(\mathbf{x})\to 0 &,&\mathrm{as} \  |\mathbf{x}|\to\infty \quad (\mathbf{x}\in\mathbf{R}^3)\,,
\end{array}\right.
\end{equation}
where $M$ denotes the total mass anomaly: $M=\int_\Omega\rho(\mathbf{x})\,\mathrm{d}\mathbf{x}$.
The gravity acceleration $\mathbf{g}$ is then given by the gradient of the potential, $\mathbf{g}:= \nabla U$. The inverse problem of gravimetry consists of recovering the anomalous density $\rho(\mathbf{x})$, with $\operatorname{supp}\rho\subset\Omega$, from measurements of the gravity acceleration $\mathbf{g}$ taken on a set $\Sigma_0 \subset \mathbf{R}^d \setminus \Omega$.

It is well known that the inverse problem of gravimetry is severely ill-posed; for instance, the same measured acceleration can correspond to vastly different density anomalies. To restore well-posedness, and in particular uniqueness, we consider a volume mass distribution where the density $\rho$ takes the form of $\rho(\mathbf{x})=f(\mathbf{x})\chi_D(\mathbf{x})$. Here, $\chi_D$ denotes the characteristic function of a domain $D$: $\chi_D(\mathbf{x})=1$, $\mathbf{x}\in D$; $\chi_D(\mathbf{x})=0$, $\mathbf{x}\notin D$. This formulation is motivated by the following well-posedness results \cite{isa90, liqia21,denzhalima25}.
\begin{thm}\label{thm1}
Let $\Omega_0$ be a convex domain with analytic (regular) boundary, $\Sigma_0\subset\partial\Omega_0$  be a nonempty hyper-surface, and $\Omega\subset\Omega_0$ be a bounded domain with connected $\mathbf{R}^d\setminus\overline{\Omega}$. 
Consider a volume mass distribution with the density function $\rho(\mathbf{x})=f(\mathbf{x})\chi_D(\mathbf{x})$; $D\subset\Omega$ denotes the domain of mass anomaly, which admits piecewise smooth boundaries, and $\chi_D$ is the characteristic function of $D$. Given the modulus of gravity acceleration, $|\nabla U|$, on $\Sigma_0$, and given $f\ge0$ in $\Omega$, the domain $D$ can be uniquely determined if one of the following constraints is satisfied:

(1) $D$ is star-shaped with respect to its center of gravity, and $f$ is constant;

(2) $D$ is convex in one direction, e.g. in $x_d$, where $x_d$ denotes a component of the spatial coordinate $\mathbf{x}=(x_1,\cdots,x_d)\in\mathbf{R}^d$, and $f$ is constant;

(3) $D$ is convex in $x_d$, $f$ does not depend on $x_d$, $f\in C(\Omega)$, and $\Omega\subset \mathrm{supp}\,f$;

(4) $D$ is convex, $f\in L_1(\Omega)$, and $0<f$ on $\Omega$.
\end{thm}

Theorem \ref{thm1} implies that the domain inverse problem of gravimetry admits uniqueness. With the density-contrast value $f(\mathbf{x})\ge0$ given a priori, the domain $D$ of the anomalous density can be uniquely determined under certain constraints. As noted in \cite{cheli24}, the domain inverse problem is a correctness class of the inverse gravimetry; we shall therefore study neural-network approaches for the domain inverse problems.

\section{Level-set physics-informed neural networks} \label{sec3}
\subsection{Level-set formulation for the domain inverse problem}
Directly representing the domain inverse problem through neural networks is challenging due to the discontinuous nature of interfaces. The evolution of interfaces can lead to vanishing or exploding gradients during training. Similar to \cite{cheli24}, we consider a level-set formulation for the domain inverse problem and use the neural network to express the continuous level-set function.

The volume-mass density function $\rho=f\chi_D$ is expressed as,
\begin{equation}
\rho(\mathbf{x})=f(\mathbf{x})\,H(\phi(\mathbf{x}))\,,
\end{equation}
where $\phi(\mathbf{x})$ is the level-set function and $H(\cdot)$ is the Heaviside function,
\begin{eqnarray}
\phi(\mathbf{x})=\left\{\begin{array}{ccc}\ge 0&,&\mathbf{x}\in\bar{D}\\ <0&,&\mathbf{x}\in\bar{D}^c \end{array} \right., \quad H(s)=\left\{\begin{array}{ccc}1&,&s\ge0\\ 0&,&s<0 \end{array} \right.\,.
\end{eqnarray}
Ideally, the level-set function $\phi$ is maintained as a continuous signed-distance function to the boundary of the domain $D$,
\begin{equation} \label{eqn3.3}
\phi(\mathbf{x})=\left\{\begin{array}{ccc}\mathrm{dist}(\mathbf{x},\partial D)&,&\mathbf{x}\in\bar{D}\\ -\mathrm{dist}(\mathbf{x},\partial D)&,&\mathbf{x}\in\bar{D}^c \end{array} \right.,
\end{equation}
such that its zero level-set $\{\mathbf{x}\mid \phi(\mathbf{x})=0\}$ depicts $\partial D$.
We use a neural network $\phi_\omega(\mathbf{x})$ to express this continuous function, where $\omega$ denotes the trainable parameters. To mitigate vanishing or exploding gradients during training, a numerical Heaviside is considered,
\begin{equation}
H_\tau(\phi)=\frac{1}{2}\left(\tanh\frac{\phi}{\tau}+1\right)\,.
\end{equation}
This formulation can be viewed as applying an activation function to the output of $\phi_\omega(\mathbf{x})$. Consequently, the neural-network representation for the volume-mass density function becomes
\begin{equation} \label{eqn_density}
\rho_\omega(\mathbf{x})=f(\mathbf{x})\,H_\tau(\phi_\omega(\mathbf{x}))\,.
\end{equation}
Rather than following the end-to-end learning approach proposed in \cite{cheli24}, which typically lacks the integration of model-based knowledge, we adopt the paradigm of physics-informed neural networks (PINNs).

\subsection{Loss function of level-set PINNs}\label{subsec_loss}
Let $U_\theta(\mathbf{x})$ denote the neural-network representation of the gravitational potential, and $\rho_\omega(\mathbf{x})$ be the density function defined in equation (\ref{eqn_density}). We propose the following loss function for level-set PINNs:
\begin{equation}
\mathcal{L}_{total}=\mathcal{L}_{pde}+\lambda_1\mathcal{L}_{bc}+\lambda_2\mathcal{L}_{far}+\lambda_3\mathcal{L}_{eik}+\lambda_4\mathcal{L}_r+\lambda_5\mathcal{L}_{\Theta}\,,
\end{equation}
where
\begin{eqnarray}
\mathcal{L}_{pde}&=& \frac{1}{N_p} \sum_{i=1}^{N_p} \left| -\Delta U_{\theta}(\mathbf{x}_i^p) - \gamma \rho_{\omega}(\mathbf{x}_i^p) \right|^2\,,  \label{eqn3.7} \\
\mathcal{L}_{bc}&=& \frac{1}{N_b} \sum_{i=1}^{N_b} \left\| \nabla U_{\theta}(\mathbf{x}_i^b) - \mathbf{g}(\mathbf{x}_i^b) \right\|_2^2\,, \\
\mathcal{L}_{far}&=& \frac{1}{N_{far}} \sum_{i=1}^{N_{far}} \left\| \nabla U_{\theta}(\mathbf{x}_i^{far}) \right\|_2^2\,, \\
\mathcal{L}_{eik}&=& \frac{1}{N_p} \sum_{i=1}^{N_p} \left( \left\| \nabla \phi_{\omega}(\mathbf{x}_i^p) \right\|_2 - 1 \right)^2\,, \\
\mathcal{L}_r&=& \frac{1}{N_p} \sum_{i=1}^{N_p} \left\| \nabla \phi_{\omega}(\mathbf{x}_i^p) \right\|_2^2 \,, \\
\mathcal{L}_{\Theta}&=& \|\theta\|_2^2+\|\omega\|_2^2\,.
\end{eqnarray}
$\big\{\mathbf{x}_i^p\big\}_{i=1}^{N_p}$, $\big\{\mathbf{x}_i^b\big\}_{i=1}^{N_b}$, and $\big\{\mathbf{x}_i^{far}\big\}_{i=1}^{N_{far}}$ denote sets of collocation points defined, respectively, in the computational domain $\Omega$, on the measurement surface $\Sigma_0$, and at a distant boundary where the far-field condition (equation (\ref{eqn4})) is applied.

The loss term $\mathcal{L}_{pde}$ enforces the partial differential equation (\ref{eqn3}), while $\mathcal{L}_{bc}$ incorporates the measurement data. The term $\mathcal{L}_{far}$ accounts for the far-field boundary condition defined in equation (\ref{eqn4}); in both $\mathbf{R}^2$ and $\mathbf{R}^3$, the potential satisfies $\nabla U(\mathbf{x})\to\mathbf{0}$ as $|\mathbf{x}|\to\infty$. $\mathcal{L}_{eik}$ and $\mathcal{L}_r$ are two regularization terms for the level-set function. The level-set function is ideally a signed distance function as shown in equation (\ref{eqn3.3}), which satisfies the eikonal equation: $|\nabla\phi(\mathbf{x})|=1$; in addition, penalizing $\mathcal{L}_r$ tends to shrink the length or area of the interface, thereby preventing structural irregularities such as burrs or sharp corners and promoting stable shape evolution \cite{denzhalima25,liqia21}. Finally, $\mathcal{L}_{\Theta}$ is a standard regularization for the trainable parameters of neural networks.

The coefficients $\lambda_i$ ($i=1,\cdots,5$) are constant weights balancing the contribution of each loss term. Determining optimal values for these hyperparameters remains a central challenge in the PINNs paradigm. In our formulation, $\lambda_4$ and $\lambda_5$ can be set to very small values, as the regularization terms $\mathcal{L}_r$ and $\mathcal{L}_{\Theta}$ are basically optional. The selection of weights $\lambda_1$ through $\lambda_3$ follows a heuristic approach based on the magnitude of their corresponding loss terms during training iterations. If a specific loss term is excessively large, its weight is adjusted upward to impose a greater penalty during training. Concrete examples of these weight selections will be provided in the numerical results section.

\subsection{Interface-aware backpropagation}
We employ the typical Adam algorithm \cite{kinba14} to optimize the neural network parameters, where a critical step involves computing the gradients of the loss function:
\begin{eqnarray}
\frac{\partial\mathcal{L}_{total}}{\partial\theta}&=&\frac{\partial\mathcal{L}_{pde}}{\partial\theta}+\lambda_1\frac{\partial\mathcal{L}_{bc}}{\partial\theta}+\lambda_2\frac{\partial\mathcal{L}_{far}}{\partial\theta}+\lambda_5\frac{\partial\mathcal{L}_{\Theta}}{\partial\theta}\,, \\
\frac{\partial\mathcal{L}_{total}}{\partial\omega}&=&\frac{\partial\mathcal{L}_{pde}}{\partial\omega}+\lambda_3\frac{\partial\mathcal{L}_{eik}}{\partial\omega}+\lambda_4\frac{\partial\mathcal{L}_r}{\partial\omega}+\lambda_5\frac{\partial\mathcal{L}_{\Theta}}{\partial\omega}\,.
\end{eqnarray}
The majority of derivatives in the above formulation are accessible via automatic differentiation within the neural network framework. However, the derivative $\frac{\partial\mathcal{L}_{pde}}{\partial\omega}$ associated with interface evolution requires a specialized computational approach.

The chain rule formally implies that
\begin{equation} \label{eqn3.15}
\frac{\partial\mathcal{L}_{pde}}{\partial\omega}=\frac{\partial\mathcal{L}_{pde}}{\partial\rho_\omega}\cdot \frac{\partial\rho_\omega}{\partial\phi_\omega} \cdot\frac{\partial\phi_\omega}{\partial\omega}\,.
\end{equation}
According to (\ref{eqn_density}),
\begin{equation} \label{eqn3.16}
\frac{\partial\rho_\omega}{\partial\phi_\omega}=f(\mathbf{x})\,\delta_\tau(\phi_\omega(\mathbf{x}))\quad \mathrm{where}\quad \delta_\tau(s)=H'_\tau(s)=\frac{1}{2\tau}\operatorname{sech}^2\left(\frac{s}{\tau}\right)\,.
\end{equation}
Here, $\tau>0$ controls the numerical thickness of the interface region. 
\begin{remark} \label{prop1}
If $\tau>0$ is chosen too large, the level-set formulation (\ref{eqn_density}) fails to depict the sharp interface of the volume-mass density.
\end{remark}
\begin{remark} \label{prop2}
If $\tau>0$ is chosen too small, it leads to the problem of exploding or vanishing gradients.
\end{remark}
\noindent Remark \ref{prop1} is explicit. Let us explain Remark \ref{prop2}. Given that $\operatorname{sech}(x)=\frac{2}{e^x+e^{-x}}$, it follows that $\delta_\tau(0)\to\infty$ as $\tau\to0$, which induces exploding gradients at the zero level-set $\{\mathbf{x}\mid \phi_\omega(\mathbf{x})=0\}$. Concurrently,
\begin{equation} \label{eqn3.17}
\operatorname{lim}_{|\frac{\phi_\omega}{\tau}|\to\infty} \delta_\tau(\phi_\omega)=0\,.
\end{equation}
That is to say, the derivative for interface evolution is concentrated within the interface region $\{\mathbf{x}\mid\phi_\omega(\mathbf{x})\sim O(\tau) \}$. If $\tau>0$ is too small, very few collocation points $\big\{\mathbf{x}_i^p\big\}_{i=1}^{N_p}$ will fall within this narrow region, leading to vanishing gradients during training. In summary, Remark \ref{prop1} and Remark \ref{prop2} state that the selection of $\tau$ presents a dilemma.

Motivated by the mechanism of interface evolution in level-set methods, we propose the following strategy for calculating $\frac{\partial\mathcal{L}_{pde}}{\partial\omega}$. First, during forward propagation, a small value for $\tau>0$ is adopted in formulation (\ref{eqn_density}) to preserve a sharp interface in the density distribution. Conversely, during backpropagation, we redefine the derivative of the numerical Heaviside as:
\begin{equation} \label{eqn3.18}
H'_\tau(s):=\delta_{\tilde{\tau}}(s) \quad \mathrm{yielding}\quad \frac{\partial\rho_\omega}{\partial\phi_\omega}=f(\mathbf{x})\,\delta_{\tilde{\tau}}(\phi_\omega(\mathbf{x}))\,;
\end{equation}
by setting $\tilde{\tau}>\tau$, for example, $\tilde{\tau}=c\tau$ with $c>1$, it creates a broader support region to drive the interface evolution. 
This approach enables an interface-aware backpropagation mitigating the issues of exploding and vanishing gradients, while maintaining the compact shape and sharp interface of the volume-mass density.

\begin{lem} \label{lem1}
Given that $\tau>0$ and $c>1$, it holds that $\left| H_{c\tau}(s) - H_\tau(s) \right|<2 e^{-2|s|/c\tau}$, $\forall s\in\mathbf{R}$.
\end{lem}
\begin{proof}
First, 
\[
0<1-\tanh(s)=\frac{2}{e^{2s}+1}< 2e^{-2s}\,, \ \ \forall s\ge0 \,.
\]
It implies that $\forall s\ge0$, $\tau>0$ and $c>1$,
\begin{eqnarray}
|H_{c\tau}(s) - H_\tau(s)| &=& \frac{1}{2} \left| \tanh\left(\frac{s}{c\tau}\right) - \tanh\left(\frac{s}{\tau}\right) \right|= \frac{1}{2} \left| \left(1 - \tanh\frac{s}{c\tau}\right) - \left(1 - \tanh\frac{s}{\tau}\right) \right| \nonumber \\
&\le&  \frac{1}{2} \left| 1 - \tanh\frac{s}{c\tau} \right| + \frac{1}{2} \left| 1 - \tanh\frac{s}{\tau} \right| \nonumber \\
&<& e^{-2s/c\tau} + e^{-2s/\tau} <2 e^{-2s/c\tau}\,. \label{lem1_1}
\end{eqnarray}
Then $\forall s<0$,
\begin{equation} \label{lem1_2}
|H_{c\tau}(s) - H_\tau(s)|=|H_{c\tau}(-s) - H_\tau(-s)|<2 e^{2s/c\tau}\,.
\end{equation}
Combining (\ref{lem1_1}) and (\ref{lem1_2}) completes the proof.
\end{proof}

\begin{proposition} \label{prop3}
$\mathcal{L}_{pde}$ is the PDE loss term defined in (\ref{eqn3.7}). Let $\tilde{\mathcal{L}}_{pde}$ denote the modified loss term obtained by replacing $\tau$ with $\tilde{\tau}=c\tau$ ($c>1$) in the formulation of $\rho_\omega$. It holds that
\begin{equation} \label{eqn3.21}
|\tilde{\mathcal{L}}_{pde}-\mathcal{L}_{pde}| \le \frac{\gamma M_1 M_2}{N_p} \sum_{i=1}^{N_p} \left| H_{c\tau}\big(\phi_\omega(\mathbf{x}_i^p)\big) - H_\tau\big(\phi_\omega(\mathbf{x}_i^p)\big) \right|\,,
\end{equation}
where
\[
M_1=\sup_i |f(\mathbf{x}_i^p)|\,,\quad M_2=\sup_i \left|-2\Delta U_{\theta}(\mathbf{x}_i^p) - \gamma \rho_{\omega}(\mathbf{x}_i^p;\tau) - \gamma \rho_{\omega}(\mathbf{x}_i^p;c\tau)\right|\,.
\]
In particular, if the collocation points $\mathbf{x}_i^p$ are uniformly sampled in the bounded domain $\Omega$, and $\phi_\omega$ satisfies the signed distance property (\ref{eqn3.3}) near the smooth interface $\partial D$, It follows that
\begin{equation}  \label{eqn3.22}
|\tilde{\mathcal{L}}_{pde} - \mathcal{L}_{pde}| \le \mathcal{K} \cdot \tau(c - 1) + O(\tau^3) \quad \mathrm{as}\ N_p\to\infty \,,
\end{equation}
where 
\[
\mathcal{K}=\gamma M_1 M_2\,\mathrm{ln}2\cdot \frac{\mathrm{Area}(\partial D)}{\mathrm{Volume}(\Omega)} \  \  \mathrm{is\ a\ constant\ independent\ of}\ \tau.
\]
\end{proposition}
\begin{proof}
The estimate (\ref{eqn3.21}) is explicit:
\begin{equation*}
\begin{aligned}
|\tilde{\mathcal{L}}_{pde}-\mathcal{L}_{pde}| 
&= \frac{\gamma}{N_p} \left| \sum_{i=1}^{N_p} \left(-2\Delta U_{\theta}(\mathbf{x}_i^p) - \gamma \rho_{\omega}(\mathbf{x}_i^p;\tau) - \gamma \rho_{\omega}(\mathbf{x}_i^p;c\tau)\right)  \big(\rho_{\omega}(\mathbf{x}_i^p;c\tau) - \rho_{\omega}(\mathbf{x}_i^p; \tau) \big) \right| \\
&\le \frac{\gamma}{N_p}  \sum_{i=1}^{N_p} \left|-2\Delta U_{\theta}(\mathbf{x}_i^p) - \gamma \rho_{\omega}(\mathbf{x}_i^p;\tau) - \gamma \rho_{\omega}(\mathbf{x}_i^p;c\tau)\right| \cdot | f(\mathbf{x}_i^p) | \\
&\hspace{45pt} \cdot \left| H_{c\tau}\big(\phi_\omega(\mathbf{x}_i^p)\big) - H_\tau\big(\phi_\omega(\mathbf{x}_i^p)\big) \right| \\
&\le \frac{\gamma M_1 M_2}{N_p} \sum_{i=1}^{N_p} \left| H_{c\tau}\big(\phi_\omega(\mathbf{x}_i^p)\big) - H_\tau\big(\phi_\omega(\mathbf{x}_i^p)\big) \right|\,.
\end{aligned}
\end{equation*}

Next, given that $\mathbf{x}_i^p$ are uniformly sampled in $\Omega$ and $\phi_\omega$ satisfies the signed distance property, we show (\ref{eqn3.22}). As $N_p\to\infty$, the average sum becomes an integral:
\begin{equation*}
\frac{1}{N_p} \sum_{i=1}^{N_p} \left| H_{c\tau}\big(\phi_\omega(\mathbf{x}_i^p)\big) - H_\tau \big(\phi_\omega(\mathbf{x}_i^p)\big) \right|
\ \xrightarrow[N_p\to\infty]{} \frac{1}{\mathrm{Volume}(\Omega)} \int_{\Omega} \left| H_{c\tau}(\phi_\omega(\mathbf{x})) - H_\tau(\phi_\omega(\mathbf{x})) \right| \mathrm{d}\mathbf{x}\,.
\end{equation*}
Denote
\begin{equation}
I=\int_{\Omega} \left| H_{c\tau}(\phi_\omega(\mathbf{x})) - H_\tau(\phi_\omega(\mathbf{x})) \right| \mathrm{d}\mathbf{x}\,,
\end{equation}
and decompose $\Omega$ into a near-interface region and a far-field region: $\Omega=\Omega_\mu \bigcup \left(\Omega\!\setminus\!\Omega_\mu\right)$, where
\begin{equation}
\Omega_\mu=\{\mathbf{x}\in \Omega: |\phi_\omega(\mathbf{x})|<\mu\},\quad\mathrm{and}\ \ \mu=\tau^{1/2}\,.
\end{equation}
Then $I=I_1+I_2$, with
\begin{equation}
I_1=\int_{\Omega\setminus\Omega_\mu} \left| H_{c\tau}(\phi_\omega(\mathbf{x})) - H_\tau(\phi_\omega(\mathbf{x})) \right| \mathrm{d}\mathbf{x}\,,\quad  I_2=\int_{\Omega_\mu} \left| H_{c\tau}(\phi_\omega(\mathbf{x})) - H_\tau(\phi_\omega(\mathbf{x})) \right| \mathrm{d}\mathbf{x}\,.
\end{equation}
According to Lemma \ref{lem1},
\begin{equation}
\left| H_{c\tau}(\phi_\omega(\mathbf{x})) - H_\tau(\phi_\omega(\mathbf{x})) \right|<2 e^{-2 |\phi_\omega(\mathbf{x}) |/c\tau}\,.
\end{equation}
In $\Omega\setminus\Omega_\mu$, we have $|\phi_\omega(\mathbf{x})|\ge\mu$, so that
\begin{equation}
\left| H_{c\tau}(\phi_\omega(\mathbf{x})) - H_\tau(\phi_\omega(\mathbf{x})) \right|<2 e^{-2 \mu /c\tau}\,.
\end{equation}
It implies that
\begin{equation} \label{eqn3.28}
I_1\le 2 e^{-2 \mu /c\tau}\cdot \mathrm{Volume}(\Omega) = 2 e^{-2/(c\sqrt{\tau})}\cdot \mathrm{Volume}(\Omega) =o(\tau^n), \quad \forall n>0\,.
\end{equation}

In the near-interface region $\Omega_\mu$, $\phi_\omega$ satisfies the signed-distance property, therefore it is Lipschitz continuous and meets $|\nabla\phi_\omega(\mathbf{x})|=1$ almost everywhere. Then we can apply the co-area formula to $I_2$,
\begin{eqnarray}
I_2&=&\int_{\Omega_\mu} \left| H_{c\tau}(\phi_\omega(\mathbf{x})) - H_\tau(\phi_\omega(\mathbf{x})) \right|\cdot |\nabla\phi_\omega(\mathbf{x})| \mathrm{d}\mathbf{x} \nonumber \\
&=& \int_{-\mu}^{\mu} \left( \int_{\Gamma_s} \left| H_{c\tau}(\phi_\omega(\mathbf{x})) - H_\tau(\phi_\omega(\mathbf{x})) \right|  \mathrm{d}\mathcal{S} \right) \mathrm{d}s\,, \label{eqn3.29}
\end{eqnarray}
where $\Gamma_s = \{ \mathbf{x} \in \Omega : \phi_\omega(\mathbf{x}) = s \}$. Rearranging (\ref{eqn3.29}) yields
\begin{equation}\label{eqn3.30}
I_2= \int_{-\mu}^{\mu}\left| H_{c\tau}(s) - H_\tau(s) \right|  \left( \int_{\Gamma_s} 1 \mathrm{d} \mathcal{S} \right) \mathrm{d}s = \int_{-\mu}^{\mu} G(s)\, \mathrm{Area}(\Gamma_s) \,\mathrm{d}s\,.
\end{equation}
Here,  for simplicity of notation we denote 
\begin{equation}
G(s):=\left| H_{c\tau}(s) - H_\tau(s) \right|.
\end{equation}
In (\ref{eqn3.30}), we expand the area of the parallel surfaces $\Gamma_s$ near the smooth interface $\partial D= \{ \mathbf{x} \in \Omega : \phi_\omega(\mathbf{x}) = 0 \}$,
\begin{equation}\label{eqn3.32}
\mathrm{Area}(\Gamma_s) = \mathrm{Area}(\partial D) + s \int_{\partial D} \kappa(\mathbf{x})\, \mathrm{d}\mathcal{S} + O(s^2)\,,
\end{equation}
where $\kappa(\mathbf{x})$ is the mean curvature (or the standard curvature in $\mathbf{R}^2$). Substituting (\ref{eqn3.32}) into (\ref{eqn3.30}) leads to
\begin{equation}\label{eqn3.33}
I_2 = \int_{-\mu}^{\mu} G(s) \left( \mathrm{Area}(\partial D) + s \int_{\partial D} \kappa(\mathbf{x})\, \mathrm{d}\mathcal{S} + O(s^2) \right) \mathrm{d}s\,.
\end{equation}
(i) The 1st part of (\ref{eqn3.33}):
\begin{eqnarray}
I_{2,1}=\int_{-\mu}^{\mu} G(s) \,\mathrm{Area}(\partial D)\, \mathrm{d}s&=&\mathrm{Area}(\partial D)\cdot 2\int_0^{\mu}\left| H_{c\tau}(s) - H_\tau(s) \right| \, \mathrm{d}s \nonumber \\
&=& \mathrm{Area}(\partial D)\cdot \int_0^{\mu}\left( \tanh \frac{s}{\tau} - \tanh \frac{s}{c\tau} \right) \,\mathrm{d}s  \nonumber \\
 &= & \mathrm{Area}(\partial D)\cdot\left[ \tau \ln\left(\cosh \frac{s}{\tau}\right) - c\tau \ln\left(\cosh \frac{s}{c\tau}\right) \right]_0^\mu \nonumber \\
 &= & \mathrm{Area}(\partial D)\cdot\left[ \tau \ln\left(\cosh \frac{\mu}{\tau}\right) - c\tau \ln\left(\cosh \frac{\mu}{c\tau}\right) \right]\,.
\end{eqnarray}
Since
\begin{equation}
\ln\left(\cosh x \right)= x-\ln2+\ln\left(1+e^{-2x}\right)=x-\ln2+O\left(e^{-2x}\right)\,,\quad\mathrm{as}\ \ x\to+\infty\,,
\end{equation}
it holds that
\begin{eqnarray}
I_{2,1}=\int_{-\mu}^{\mu} G(s) \,\mathrm{Area}(\partial D)\, \mathrm{d}s &=& \mathrm{Area}(\partial D)\cdot \left[ (\mu - \tau \ln 2) - (\mu - c\tau \ln 2) + O\left(e^{-2\mu/c\tau}\right) \right] \nonumber \\
&=& \mathrm{Area}(\partial D)\cdot \tau(c-1) \ln 2 + O\left(e^{-2/(c\sqrt{\tau})}\right) \nonumber \\
&=& \mathrm{Area}(\partial D)\cdot \tau(c-1) \ln 2 + o(\tau^n)\,,\quad \forall n>0\,. \label{eqn3.36}
\end{eqnarray}
(ii) The 2nd part of (\ref{eqn3.33}):
\begin{equation} \label{eqn3.37}
I_{2,2}=\int_{-\mu}^{\mu} G(s) s \left( \int_{\partial D} \kappa(\mathbf{x})\, \mathrm{d}\mathcal{S} \right) \mathrm{d}s = \int_{\partial D} \kappa(\mathbf{x})\, \mathrm{d}\mathcal{S}\cdot \int_{-\mu}^{\mu} G(s) s \, \mathrm{d}s =0\,,
\end{equation}
given that the integrand $G(s) s$ is an odd function.

\noindent(iii) The residual part of (\ref{eqn3.33}):
\begin{equation}
I_{2,res}=\int_{-\mu}^{\mu} G(s)\, O(s^2) \, \mathrm{d}s = 2 \int_0^{\mu} G(s)\, O(s^2) \, \mathrm{d}s \le M_{res} \int_0^{\mu} \left( \tanh \frac{s}{\tau} - \tanh \frac{s}{c\tau}\right) s^2 \, \mathrm{d}s \,,
\end{equation}
where $M_{res}$ denotes a positive constant. Perform the change of variables $v = \frac{s}{\tau}$, and recall that $\mu=\tau^{1/2}$,
\begin{eqnarray}
I_{2,res}&\le& M_{res} \int_0^{\tau^{-1/2}} \left( \tanh v - \tanh \frac{v}{c}\right) \left(\tau^2 v^2\right) \, \tau\mathrm{d}v \nonumber \\
&=& \tau^3\cdot M_{res} \int_0^{\tau^{-1/2}} v^2 \left( \tanh v - \tanh \frac{v}{c}\right) \mathrm{d}v \nonumber \\
&\le& \tau^3\cdot M_{res} \int_0^{\infty} v^2 \left( \tanh v - \tanh \frac{v}{c}\right) \mathrm{d}v \,. \label{eqn3.39}
\end{eqnarray}
Denote
\begin{eqnarray}
I_\infty = \int_0^{\infty} v^2 \left( \tanh v - \tanh \frac{v}{c}\right) \mathrm{d}v&= &\int_0^{\infty} v^2 \left(  \left( 1 - \frac{2}{e^{2v} + 1} \right) - \left( 1 - \frac{2}{e^{2v/c} + 1} \right) \right) \mathrm{d}v \nonumber \\
&=& 2 \int_0^{\infty} v^2 \left(  \frac{1}{e^{2v/c} + 1} - \frac{1}{e^{2v} + 1} \right) \mathrm{d}v \,. \label{eqn3.40}
\end{eqnarray}
We can consider the Fermi-Dirac integral (see, e.g., formula 3.411 - 3 in \cite{graryz07}):
\begin{equation} \label{eqn3.41}
\int_0^\infty \frac{x^{n-1}}{e^{ax} + 1}\, \mathrm{d}x = \frac{1}{a^n}(1 - 2^{1-n}) \Gamma(n) \zeta(n), \qquad Re (a) > 0\,,\ \ Re (n) > 0\,,
\end{equation}
where $\Gamma(n)  = \int_0^\infty x^{n-1} e^{-x} \, \mathrm{d}x $ is the Gamma function, and $\zeta(n)$ is the Riemann Zeta function. We comment that $(1 - 2^{1-n}) \zeta(n)= \sum_{k=1}^\infty \frac{(-1)^{k-1}}{k^n}$ is the Dirichlet Eta function, which converges for all $Re(n) > 0$. 
Utilizing (\ref{eqn3.41}) in (\ref{eqn3.40}) leads to: 
\begin{equation} \label{eqn3.42}
I_\infty = 2  \left( \frac{ \frac{3}{4}\,\Gamma(3) \zeta(3)}{(2/c)^3}- \frac{ \frac{3}{4}\,\Gamma(3) \zeta(3)}{8}  \right) =  \frac{3}{8}(c^3 - 1)\, \zeta(3)\,,
\end{equation}
where $\Gamma(3)=2$, and $\zeta(3) \approx 1.202$ is Apéry's constant.

Substituting (\ref{eqn3.42}) into (\ref{eqn3.39}), we have that
\begin{equation} \label{eqn3.43}
I_{2,res} \le \tau^3\cdot M_{res}\, \frac{3}{8}(c^3 - 1)\, \zeta(3) =O(\tau^3)\,.
\end{equation}
Combining (\ref{eqn3.36}), (\ref{eqn3.37}) and (\ref{eqn3.43}) leads to
\begin{equation}\label{eqn3.44}
I_2= \mathrm{Area}(\partial D)\cdot \tau(c-1) \ln 2 + O(\tau^3)\,.
\end{equation}
Combining (\ref{eqn3.28}) and (\ref{eqn3.44}) yields that
\begin{equation}
I=I_1+I_2=\mathrm{Area}(\partial D)\cdot \tau(c-1) \ln 2 + O(\tau^3)\,.
\end{equation}
It completes the error estimate (\ref{eqn3.22}) by considering (\ref{eqn3.21}).
\end{proof}

In the interface-aware backpropagation, we redefine the derivative of $H_\tau$ according to equation (\ref{eqn3.18}), i.e., $H'_\tau(s):=\delta_{\tilde{\tau}}(s)$. Using this new derivative in an optimization algorithm (e.g., Adam) essentially minimizes the PDE loss term $\tilde{\mathcal{L}}_{pde}$. Proposition \ref{prop3} shows that $\tilde{\mathcal{L}}_{pde}$ closely approximates $\mathcal{L}_{pde}$ for small $\tau$. Therefore, minimizing $\tilde{\mathcal{L}}_{pde}$ effectively minimizes $\mathcal{L}_{pde}$ as well.


\subsection{Adaptive refinement of collocation points near interfaces}
As illustrated by equations (\ref{eqn3.15}) - (\ref{eqn3.17}), the update of $\omega$ via $\mathcal{L}_{pde}$ is primarily driven by collocation points within the interface region. Allocating more points in this region enables $\phi_\omega$ to better characterize the interface structure. Leveraging the mesh-free nature of the neural network representation and the fact that the interface is described by the zero level-set, we propose a simple strategy for adaptive refinement of collocation points near interfaces. One can optionally use it to improve the interface characterization.

The procedure begins by defining a static set of candidate points, $S$, within the computational domain $\Omega$; these may be generated, for example, via a uniform random distribution. During an iteration, points proximal to the interface are sampled from $S$ to form the subset
\begin{equation}
I_{\phi_\omega}=\{\mathbf{x}\in S: |\phi_\omega(\mathbf{x})|<C_{\mathrm{refine}}\},
\end{equation} 
where the hyper-parameter $C_{\mathrm{refine}}$ determines the width of the refinement zone. These interface points, $I_{\phi_\omega}$, are subsequently merged into the full set of collocation points for the domain, which we denote (with slight notation overlap) as $\big\{\mathbf{x}_i^p\big\}_{i=1}^{N_p}$. This adaptive refinement strategy directly influences the evaluation of $\mathcal{L}_{pde}$, $\mathcal{L}_{eik}$ and $\mathcal{L}_r$ in the total loss function.

\subsection{Choices of neural networks}
We study the choice of neural networks $\phi_\omega(\mathbf{x})$ and $U_\theta(\mathbf{x})$ for the level-set function $\phi(\mathbf{x})$ and the gravitational potential $U(\mathbf{x})$, respectively.
\begin{proposition} \label{prop4}
Let $\Omega\subset\mathbf{R}^n$ be a bounded domain (open set), and let $D\subset\Omega$ denote the domain of mass anomaly, which admits piecewise smooth boundaries. The signed-distance function $\phi(\mathbf{x})$ defined in (\ref{eqn3.3}) satisfies $\phi(\mathbf{x})\in W^{1,\infty}(\Omega)$.
\end{proposition}
\begin{proposition} \label{prop5}
Let $\Omega_0$ be a bounded domain and $\Omega\subset\Omega_0$. Given that the density $\rho$ with $\operatorname{supp}\rho\subset\Omega$ satisfies $\rho\in L^\infty(\Omega)$, the gravity potential $U$ satisfies $U\in W^{1,\infty}(\Omega_0)$.
\end{proposition}

Proposition \ref{prop4} is a typical property of signed-distance functions. In fact, $\phi(\mathbf{x})$ satisfies the eikonal equation $|\nabla\phi(\mathbf{x})|=1$ almost everywhere in $\Omega$. 

We provide a brief proof of Proposition \ref{prop5}.
\begin{proof}
(Proposition \ref{prop5}) According to (\ref{eqn1}),
\[
|U(\mathbf{y})| \leq \gamma \|\rho\|_{L^\infty(\Omega)} \int_\Omega |K(\mathbf{y}, \mathbf{x})| \, \mathrm{d}\mathbf{x},\quad  |\nabla U(\mathbf{y})| \leq \gamma \|\rho\|_{L^\infty(\Omega)} \int_\Omega |\nabla_{\mathbf{y}}K(\mathbf{y}, \mathbf{x})| \, \mathrm{d}\mathbf{x}.
\]
Since $\Omega_0$ is bounded, there exists $r_0>\frac{1}{2}$ such that $\Omega\subset\Omega_0\subset B(\mathbf{0}, r_0)$ (a ball of radius $r_0$).

(1) In $\mathbf{R}^2$, $K(\mathbf{y},\mathbf{x})=-\frac{1}{2\pi}\mathrm{ln}|\mathbf{y}-\mathbf{x}|$ and $\nabla_{\mathbf{y}}K(\mathbf{y}, \mathbf{x})=-\frac{1}{2\pi}\frac{\mathbf{y}-\mathbf{x}}{|\mathbf{y}-\mathbf{x}|^2}$. It follows that $\forall\mathbf{y}\in\Omega_0$,
\[
|U(\mathbf{y})| \leq \gamma \|\rho\|_{L^\infty(\Omega)} \int_{|\mathbf{z}| \le 2 r_0} \frac{1}{2\pi}\big|\mathrm{ln}|\mathbf{z}|\big| \, \mathrm{d}\mathbf{z} = \gamma \|\rho\|_{L^\infty(\Omega)} \int_0^{2 r_0}|\mathrm{ln}r|r\,\mathrm{d}r\,.
\]
Since $2 r_0>1$, $\int_0^{2 r_0}|\mathrm{ln}r|r\,\mathrm{d}r=2r_0^2\mathrm{ln}(2r_0)-r_0^2+\frac{1}{2}$, which implies that
$
|U(\mathbf{y})| \leq \gamma \|\rho\|_{L^\infty(\Omega)}\left(2 r_0^2\mathrm{ln}(2 r_0)-r_0^2+\frac{1}{2}\right)\,.
$
Similarly, $\forall\mathbf{y}\in\Omega_0$,
\[
 |\nabla U(\mathbf{y})| \leq \gamma \|\rho\|_{L^\infty(\Omega)} \int_{|\mathbf{z}| \le 2 r_0}\frac{1}{2\pi |\mathbf{z}|} \, \mathrm{d}\mathbf{z}=\gamma \|\rho\|_{L^\infty(\Omega)} \int_0^{2 r_0}\mathrm{d}r=2r_0\gamma\|\rho\|_{L^\infty(\Omega)}\,.
\]

(2) In $\mathbf{R}^3$, $K(\mathbf{y},\mathbf{x})=\frac{1}{4\pi|\mathbf{y}-\mathbf{x}|}$ and $\nabla_{\mathbf{y}}K(\mathbf{y}, \mathbf{x})=-\frac{1}{4\pi}\frac{\mathbf{y}-\mathbf{x}}{|\mathbf{y}-\mathbf{x}|^3}$. It follows that $\forall\mathbf{y}\in\Omega_0$,
\begin{align*}
|U(\mathbf{y})| \leq \gamma \|\rho\|_{L^\infty(\Omega)} \int_{|\mathbf{z}| \le 2 r_0} \frac{1}{4\pi |\mathbf{z}|}\, \mathrm{d}\mathbf{z} = \gamma \|\rho\|_{L^\infty(\Omega)} \int_0^{2r_0}r\,\mathrm{d}r=2r_0^2 \gamma \|\rho\|_{L^\infty(\Omega)}\,, \\
|\nabla U(\mathbf{y})| \leq \gamma \|\rho\|_{L^\infty(\Omega)} \int_{|\mathbf{z}| \le 2r_0} \frac{1}{4\pi |\mathbf{z}|^2}\, \mathrm{d}\mathbf{z} = \gamma \|\rho\|_{L^\infty(\Omega)} \int_0^{2r_0} \mathrm{d}r=2r_0 \gamma \|\rho\|_{L^\infty(\Omega)}\,. 
\end{align*}
\end{proof}

Propositions \ref{prop4} and \ref{prop5} suggest seeking neural network representations for $\phi(\mathbf{x})$ and $U(\mathbf{x})$ within the Sobolev spaces $W^{1,\infty}(\Omega)$ and $W^{1,\infty}(\Omega_0)$, respectively. To achieve this, we employ fully connected neural networks of the form 
\begin{equation} \label{eqn:fnn}
	\mathcal{N}= \mathcal{L}_L\circ\sigma\circ\mathcal{L}_{L-1}\circ \cdots \sigma\circ\mathcal{L}_l\circ \cdots \sigma \circ\mathcal{L}_1\,,
\end{equation}
where $\sigma$ denotes the activation function, $\mathcal{L}_l(\mathbf{z})=\mathbf{W}_l\mathbf{z}+\mathbf{b}_l$ denotes the affine mapping of the $l$-th layer, and $L$ represents the total number of layers. 

By convention, we name neural networks based on their activation functions. $\mathcal{N}$ is referred to as a ReLU neural network if $\sigma$ is the ReLU function, and a $\tanh$ neural network if $\sigma$ is the $\tanh$ function. The following results \cite{yar17, zhaluzha24} indicate that both ReLU and $\tanh$ networks are well-suited for approximating functions in $W^{1,\infty}$.

\subsubsection{Approximation by ReLU networks} \label{subsubsec3.5.1}
First, we invoke the upper-bound result from \cite{yar17} to establish the approximation capacity.
\begin{thm}\cite[Theorem~1]{yar17}\label{prop:relu_bounds_en}
	Let $F_{d,n}$ denote the unit ball of $W^{n,\infty}([0,1]^d)$:
	\[
		F_{d,n}=\bigl\{\varphi\in W^{n,\infty}([0,1]^d):\|\varphi\|_{W^{n,\infty}}\le1\bigr\}.
	\]
	Then, there exists a constant $c_{0}=c_{0}(d,n)>0$, such that $\forall\,\varepsilon\in(0,1)$ and $\forall\,\varphi\in F_{d,n}$, there exists a ReLU neural network $\mathcal{N}$ of the form (\ref{eqn:fnn}) with depth
	\begin{equation} \label{eqn:relu_depth}
		L\le c_{0}\bigl(\ln(1/\varepsilon)+1\bigr),
	\end{equation}
	and a total number of weights and computation units
	\begin{equation} \label{eqn:relu_weight_upper}
		W_{\mathrm{tot}}\le c_{0}\,\varepsilon^{-d/n}\bigl(\ln(1/\varepsilon)+1\bigr),
	\end{equation}
	such that
	\begin{equation} \label{eqn:relu_approx}
		\|\varphi-\mathcal{N}\|_{L^\infty([0,1]^d)}\le\varepsilon.
	\end{equation}
\end{thm}

Since any function $\varphi\in W^{n,\infty}([0,1]^d)$ can be normalized as $\frac{\varphi}{C} \in F_{d,n}$ with $C := \|\varphi\|_{W^{n,\infty}([0,1]^d)}$,  Theorem~\ref{prop:relu_bounds_en} effectively establishes the approximation capacity of ReLU networks for $W^{n,\infty}$ functions. A natural converse question is: given a network $\mathcal{N}$ with a depth $L$ and a parameter budget $W_{\mathrm{tot}}$, what is the best guaranteed error? We address this by inverting the bounds \eqref{eqn:relu_depth}--\eqref{eqn:relu_weight_upper} to derive explicit error estimates in terms of $L$ and $W_{\mathrm{tot}}$.

\medskip

\noindent\textbf{Inversion of the depth bound.}
Inequality (\ref{eqn:relu_depth}), $L\le c_{0}(\ln(1/\varepsilon)+1)$, yields
\begin{equation}\label{eqn:eps_from_depth}
	\varepsilon \le e^{1 - L/c_{0}},
\end{equation}
indicating the guaranteed upper bound on the approximation error in terms of the network depth.

\medskip

\noindent\textbf{Inversion of the parameter bound.}
\begin{corollary}\label{cor:ReLu_approx}
The guaranteed upper bound on the approximation error $\varepsilon \in (0,1)$ of ReLU networks for $W^{n,\infty}$ functions is:
\begin{equation}\label{eqn:eps_from_W}
    \varepsilon \le \left[ \frac{n c_{0} \mathrm{W_{Lambert}}\left( \frac{W_{\mathrm{tot}} d}{n c_{0}} e^{d/n} \right)}{W_{\mathrm{tot}} d} \right]^{n/d},
\end{equation}
where $\mathrm{W_{Lambert}}(\cdot)$ denotes the Lambert $\mathrm{W}$ function \cite{corless1996lambert} defined by the relation
\begin{equation}\label{eqn:LambertW}
\mathrm{W_{Lambert}}(x)e^{\mathrm{W_{Lambert}}(x)}= x. 
\end{equation}
\end{corollary}
\begin{proof}
Starting from (\ref{eqn:relu_weight_upper}), $W_{\mathrm{tot}} \le c_{0}\,\varepsilon^{-d/n}(\ln(1/\varepsilon)+1)$ for $\varepsilon \in (0,1)$, we set $k = d/n$ and $u =(1/\varepsilon)^k$, so that the inequality becomes 
\[
W_{\mathrm{tot}} \le c_{0}u\left(\frac{1}{k}\ln u + 1\right)=c_{0}u \left(\frac{1}{k}\ln(ue^k)\right)\,.
\] 
It implies that
\[
\frac{kW_{\mathrm{tot}}}{c_{0}}\,e^k \;\le\; (ue^k)\ln(ue^k).
\]
Denoting $X:=\frac{kW_{\mathrm{tot}}}{c_{0}}\,e^k>0$, the equation $X=Y\ln Y$ (for $Y>1$ and so $X>0$) has a unique solution $Y = e^{\mathrm{W_{Lambert}}(X)}$, where $\mathrm{W_{Lambert}}$ is the Lambert $\mathrm{W}$ function \cite{corless1996lambert} defined by (\ref{eqn:LambertW}). Moreover, since $f(Y) = Y\ln Y$ is strictly increasing for $Y>1$, the inequality 
\[X=Y\ln Y\le (ue^k)\ln(ue^k)
\]
implies $Y\le ue^k$, viz., $Y=e^{\mathrm{W_{Lambert}}(X)}\le ue^k$. According to (\ref{eqn:LambertW}), $e^{\mathrm{W_{Lambert}}(X)}=\frac{X}{\mathrm{W_{Lambert}}(X)}$, and we have that
\[
\frac{X}{\mathrm{W_{Lambert}}(X)} \le ue^k,
\]
which means
\begin{equation}\label{eqn:3.55}
\frac{ \frac{W_{\mathrm{tot}} d}{n c_{0}} e^{d/n}}{\mathrm{W_{Lambert}}( \frac{W_{\mathrm{tot}} d}{n c_{0}} e^{d/n})} \le (1/\varepsilon)^{d/n}e^{d/n}\,.
\end{equation}
(\ref{eqn:3.55}) leads to the error estimate (\ref{eqn:eps_from_W}) immediately.
\end{proof}

\begin{remark} \label{WLambert}
For a large parameter budget $W_{\mathrm{tot}}$, $\mathrm{W_{Lambert}}(X)$ can be evaluated according to its asymptotic behavior \cite{corless1996lambert}:
\begin{equation} \label{eqn3.56}
\mathrm{W_{Lambert}}(X) \sim \ln X - \ln\ln X\,, \qquad \mathrm{as}\ \ X\to\infty\,.
\end{equation}
\end{remark}

\subsubsection{Approximation by $\tanh$ networks} \label{subsubsec3.5.2}
We adopt the result of \cite{zhaluzha24}, which bridges ReLU networks and networks activated by a broad class of activation functions including $\tanh$.

\begin{thm}\cite[Theorem~1]{zhaluzha24}\label{prop:tanh_bridge_en}
Let $\mathcal{N}_{\mathrm{ReLU}}:\mathbf{R}^d\to\mathbf{R}^{d'}$ be a ReLU network of the form (\ref{eqn:fnn}) with width $N$ and depth $L$, where $N, L, d, d' \in \mathbf{N}^+$. Then $\forall\varepsilon > 0$ and $\forall A > 0$, there exists a $\tanh$ network $\mathcal{N}_{\tanh}:\mathbf{R}^d\to\mathbf{R}^{d'}$ of the form (\ref{eqn:fnn}) with width $3N$ and depth $2L$ such that
	\begin{equation} \label{eqn:tanh_bridge}
		\left\| \mathcal{N}_{\tanh}(\mathbf{x}) - \mathcal{N}_{\mathrm{ReLU}}(\mathbf{x})\right\|_{L^\infty([-A,A]^d)} < \varepsilon\,.
	\end{equation}
\end{thm}

\begin{remark} \label{neural_width}
The width of a neural network refers to the maximum number of neurons across its hidden layers. Specifically, consider a network $\mathcal{N}:\mathbf{R}^d\to\mathbf{R}^{d'}$ of the form (\ref{eqn:fnn}) with layers $\mathcal{L}_l(\mathbf{z})=\mathbf{W}_l\mathbf{z}+\mathbf{b}_l$, where $\mathbf{W}_l\in\mathbf{R}^{N_{l+1}\times{N_l}}$ and $\mathbf{b}_l\in\mathbf{R}^{N_{l+1}}$. It holds that $N_1=d$, $N_2,\cdots,N_L\in \mathbf{N}^+$, and $N_{l+1}=d'$. The width of the network is $N=\max\{N_2,\cdots,N_L\}$.
\end{remark}

Since the total parameter count of a fully connected network scales as $(\text{width})^2 \times \text{depth}$, the $(3,2)$ scaling in Theorem~\ref{prop:tanh_bridge_en} multiplies the parameter budget by $3^2 \times 2 = 18$. Then, combining Theorem \ref{prop:relu_bounds_en} and Theorem \ref{prop:tanh_bridge_en} immediately leads to the following result.

\begin{corollary}\label{prop:tanh_bounds_en}
	Let $F_{d,n}$ denote the unit ball of $W^{n,\infty}([0,1]^d)$:
	\[
		F_{d,n}=\bigl\{\varphi\in W^{n,\infty}([0,1]^d):\|\varphi\|_{W^{n,\infty}}\le1\bigr\}.
	\]
	Then, there exists a constant $c_{0}=c_{0}(d,n)>0$, such that $\forall\,\varepsilon\in(0,1)$ and $\forall\,\varphi\in F_{d,n}$, there exists a $\tanh$ neural network $\mathcal{N}$ of the form (\ref{eqn:fnn}) with depth
	\begin{equation} \label{eqn:tanh_depth}
		L\le 2c_{0}\bigl(\ln(2/\varepsilon)+1\bigr),
	\end{equation}
	and a total number of weights and computation units
	\begin{equation} \label{eqn:tanh_weight_upper}
		W_{\mathrm{tot}}\le 18c_{0}\,(\varepsilon/2)^{-d/n}\bigl(\ln(2/\varepsilon)+1\bigr),
	\end{equation}
	such that
	\begin{equation} \label{eqn:tanh_approx}
		\|\varphi-\mathcal{N}\|_{L^\infty([0,1]^d)}<\varepsilon.
	\end{equation}
\end{corollary}

\medskip

\noindent\textbf{Inversion of the depth bound.}
From the inequality (\ref{eqn:tanh_depth}), we obtain that
\begin{equation} \label{eqn:tanh_eps_from_depth}
	\varepsilon \leq 2e^{1 - L/(2c_0)}\,.
\end{equation}

\medskip

\noindent\textbf{Inversion of the parameter bound.}
Following the analysis in Corollary \ref{cor:ReLu_approx}, we establish the following error bound in terms of $W_{\mathrm{tot}}$.

\begin{corollary}\label{cor:tanh_approx}
The guaranteed upper bound on the approximation error $\varepsilon \in (0,1)$ of $\tanh$ networks for $W^{n,\infty}$ functions is:
\begin{equation}\label{eqn:tanh_eps_from_W}
    \varepsilon \le 2\left[ \frac{18n c_{0} \mathrm{W_{Lambert}}\left( \frac{W_{\mathrm{tot}} d}{18n c_{0}} e^{d/n} \right)}{W_{\mathrm{tot}} d} \right]^{n/d},
\end{equation}
where $\mathrm{W_{Lambert}}(\cdot)$ denotes the Lambert $\mathrm{W}$ function.
\end{corollary}

Again, for a large parameter budget $W_{\mathrm{tot}}$, the asymptotic expansion in (\ref{eqn3.56}) can be utilized to evaluate $\mathrm{W_{Lambert}}(X)$.

\subsubsection{Architectures of $\phi_\omega(\mathbf{x})$ and $U_\theta(\mathbf{x})$}

As established in Propositions~\ref{prop4} and~\ref{prop5}, both the level-set function $\phi$ and the gravitational potential $U$ possess $W^{1,\infty}$ regularity. The theoretical results in Sections~\ref{subsubsec3.5.1} and~\ref{subsubsec3.5.2} confirm that both ReLU and $\tanh$ networks have the capacity to approximate such functions.
In this work, we employ a ReLU network of the form (\ref{eqn:fnn}) for $\phi_\omega(\mathbf{\mathbf{x}})$, and a $\tanh$ network of the form (\ref{eqn:fnn}) for $U_\theta(\mathbf{\mathbf{x}})$.

In contemporary neural network practice, ReLU is typically preferred over the $\tanh$ activation function. First, unlike $\tanh$, ReLU does not saturate for large positive inputs, thereby mitigating the vanishing gradient problem. Second, the milder nonlinearity of ReLU networks facilitates more efficient optimization during training. Finally, our approximation results indicate that a $\tanh$ network requires greater width and depth to match the approximation accuracy of a ReLU network.  Consequently, we employ a ReLU network $\phi_\omega(\mathbf{x})$ to represent the level-set function $\phi(\mathbf{x})$.

For the gravitational potential network $U_\theta(\mathbf{x})$, however, the PDE loss $\mathcal{L}_{pde}$ involves the Laplacian $\Delta U_\theta$. Because ReLU is piecewise linear, its second derivative vanishes almost everywhere, rendering the Laplacian uninformative. To obtain a well-defined PDE loss via automatic differentiation, we instead employ the $\tanh$ activation function for $U_\theta(\mathbf{x})$.

\medskip

\noindent \textbf{Implementation details in numerical experiments.}

In the upcoming numerical experiments, both networks are parameterized with $L=6$ layers (5 hidden layers and 1 linear output layer), mapping a $d$-dimensional input $\mathbf{x} \in \mathbf{R}^d$ ($d=2,3$) to a scalar output. All hidden layers have a uniform width of 300 neurons.

\begin{itemize}
	\item \textbf{Architecture of $\phi_\omega$:}
	The activation is $\sigma = \mathrm{ReLU}$. The weight matrices have dimensions $\mathbf{W}_1 \in \mathbf{R}^{300 \times d}$, $\mathbf{W}_l \in \mathbf{R}^{300 \times 300}$ for $l=2,\dots,5$, and $\mathbf{W}_6 \in \mathbf{R}^{1 \times 300}$.
	
	\item \textbf{Architecture of $U_\theta$:}
	The activation is $\sigma = \tanh$. The architecture shares the same structure as $\phi_\omega$.
\end{itemize}

Both $\phi_\omega(\mathbf{x})$ and $U_\theta(\mathbf{x})$ share the same total number of trainable parameters (weights and biases), given by
\[
W_{\mathrm{tot}} = (300d+300) + 4 \times( 300^2 + 300) + (300 + 1) = 300d + 361801 \doteq 3.6 \times 10^5
\]
for $d=2,3$.

Then, from Equations (\ref{eqn:eps_from_depth}) and (\ref{eqn:eps_from_W}), we obtain an upper bound on the approximation error of the ReLU network $\phi_\omega$:
\[
\varepsilon_{\phi_\omega}=\max\left\{e^{1 - L/c_{0}}, \left[ \frac{c_{0} \mathrm{W_{Lambert}}\left( \frac{W_{\mathrm{tot}} d}{c_{0}} e^{d} \right)}{W_{\mathrm{tot}} d} \right]^{1/d}\right\}
\]
with $L=6,\,W_{\mathrm{tot}}\doteq3.6 \times 10^5$. Taking the constant $c_0=c_0(d,n)=1$, we obtain
\[
\varepsilon_{\phi_\omega}\doteq \max\left\{e^{-5}, \left[\frac{\mathrm{W_{Lambert}}\left(3.6 \times 10^5 d e^{d} \right)}{3.6 \times 10^5 d} \right]^{1/d}\right\}\,.
\]
It implies that
\begin{equation*}
\varepsilon_{\phi_\omega}\doteq \left\{
\begin{array}{lr}
\max\left\{e^{-5}, \left[\frac{\mathrm{W_{Lambert}}\left(5.32 \times 10^6 \right)}{7.2 \times 10^5} \right]^{1/2}\right\}\doteq\max\{0.0067,0.0042\}=0.0067\,,&d=2. \\
\max\left\{e^{-5}, \left[\frac{\mathrm{W_{Lambert}}\left(2.17 \times 10^7 \right)}{1.08 \times 10^6} \right]^{1/3}\right\}\doteq\max\{0.0067,0.0235\}=0.0235\,,&d=3.
\end{array}
\right.
\end{equation*}
Here, the asymptotic expansion in (\ref{eqn3.56}) is used to evaluate $\mathrm{W_{Lambert}}(X)$.

Similarly, from Equations (\ref{eqn:tanh_eps_from_depth}) and (\ref{eqn:tanh_eps_from_W}), we obtain an upper bound on the approximation error of the $\tanh$ network $U_\theta$:
\[
\varepsilon_{U_\theta}\doteq \max\left\{2e^{-2}, 2\left[\frac{18\mathrm{W_{Lambert}}\left(2 \times 10^4 d e^{d} \right)}{3.6 \times 10^5 d} \right]^{1/d}\right\}\,.
\]
It implies that
\begin{equation*}
\varepsilon_{U_\theta}\doteq \left\{
\begin{array}{lr}
\max\left\{2e^{-2}, 2\left[\frac{\mathrm{W_{Lambert}}\left(2.96 \times 10^5\right)}{4 \times 10^4} \right]^{1/2}\right\}\doteq\max\{0.27,0.032\}=0.27\,,&d=2. \\
\max\left\{2e^{-2}, 2\left[\frac{\mathrm{W_{Lambert}}\left(1.21\times 10^6 \right)}{6 \times 10^4} \right]^{1/3}\right\}\doteq\max\{0.27,0.11\}=0.27\,,&d=3.
\end{array}
\right.
\end{equation*}

In practice, both $\phi_\omega(\mathbf{x})$ and $U_\theta(\mathbf{x})$ can outperform the aforementioned theoretical upper bounds.

\section{Numerical results}
\label{sec:results}
We provide numerical examples in both 2D and 3D to illustrate the performance of level-set physics-informed neural networks.

\subsection{Network initialization via pre-training}

The gravity inverse problem is inherently ill-posed, and neural-network-based solvers exhibit pronounced sensitivity to initial parameterization. To mitigate these difficulties, we adopt a two-step pre-training strategy prior to the formal data-driven inversion: we first initialize the level-set network $\phi_{\omega}(\mathbf{x})$, and then initialize the gravitational potential network $U_{\theta}(\mathbf{x})$ based on the initialized level-set network.

\noindent \textbf{Step 1: Initialize $\phi_{\omega}(\mathbf{x})$.}
Let $\phi(\mathbf{x})$ denote the targeted initial guess of the level-set function. Then the network $\phi_{\omega}(\mathbf{x})$ is pre-trained by minimizing
\begin{equation}
	\mathcal{L}_{\Phi}(\omega) = \frac{1}{N_p} \sum_{i=1}^{N_p} \left|\phi_{\omega}(\mathbf{x}_i) - \phi(\mathbf{x}_i)\right|^2\,,
\end{equation}
where $\big\{\mathbf{x}_i^p\big\}_{i=1}^{N_p}$ denotes the set of collocation points in the computational domain $\Omega$.

For example, in our 2D experiment with $\Omega = (0,1)^2$, the targeted initial guess of the level-set function is set as
\begin{equation}
	\phi(\mathbf{x}) = 1 - \sqrt{(x-0.5)^2/0.35^2 + (z-0.5)^2/0.35^2}.
\end{equation}
Figure \ref{fig:pretrain_phi}\,(a) shows the output of the pre-trained $\phi_{\omega}(\mathbf{x})$.

For 3D experiments with $\Omega = (0,1)^3$, we set the targeted initial guess of the level-set function as
\begin{equation}
	\phi(\mathbf{x}) = 1 - \sqrt{(x-0.5)^2/0.35^2 + (y-0.5)^2/0.35^2 + (z-0.5)^2/0.35^2}.
\end{equation}
Figure \ref{fig:pretrain_phi}\,(b) shows the zero level set of the pre-trained $\phi_{\omega}(\mathbf{x})$.

\noindent \textbf{Step 2: Initialize $U_{\theta}(\mathbf{x})$.}
Subsequently, the potential network $U_{\theta}(\mathbf{x})$ is pre-trained by minimizing
\begin{equation}
	\mathcal{L}_{u}(\theta) = \frac{1}{N_p} \sum_{i=1}^{N_p} \left|-\Delta U_{\theta}(\mathbf{x}_i) - \gamma\rho_{\omega}(\mathbf{x}_i)\right|^2\,,
\end{equation}
where $\big\{\mathbf{x}_i^p\big\}_{i=1}^{N_p}$ denotes the set of collocation points in the computational domain $\Omega$. It enforces the governing PDE in a least-squares sense. Figure \ref{fig:pretrain_U} shows the output of the pre-trained $U_{\theta}(\mathbf{x})$ for both 2D and 3D experiments.

\begin{figure}[htbp]
	\centering
	\begin{subfigure}[b]{0.3\textwidth}
		\centering
		\includegraphics[width=\textwidth]{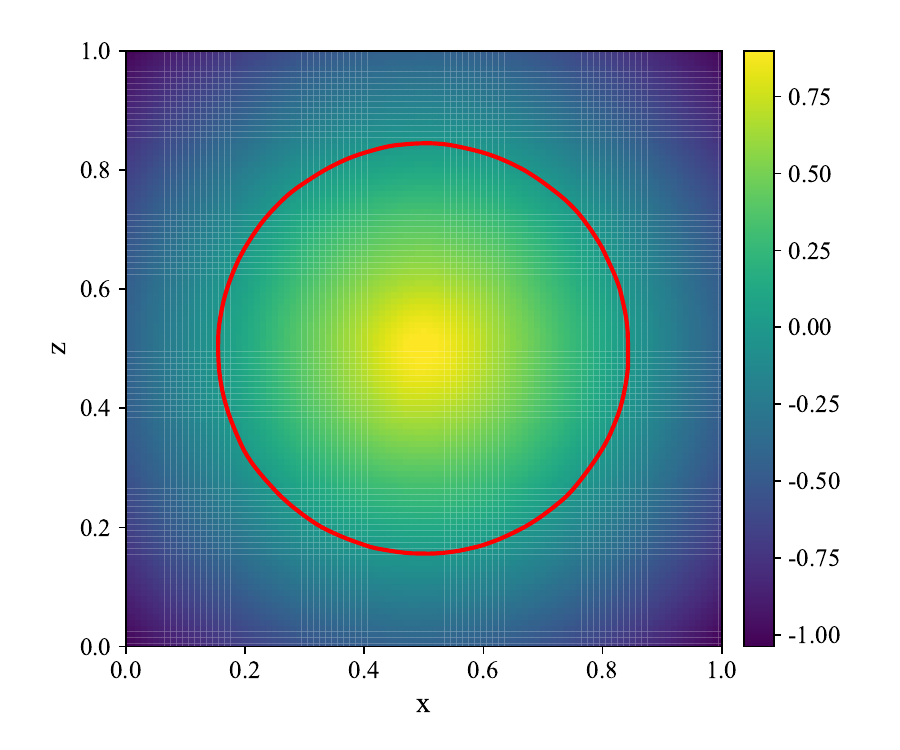}
		\caption{}
		\label{fig:pretrain_2d_phi}
	\end{subfigure}\hspace{0.02\textwidth}
	 \begin{subfigure}[b]{0.36\textwidth}
		\centering
		\includegraphics[width=\textwidth]{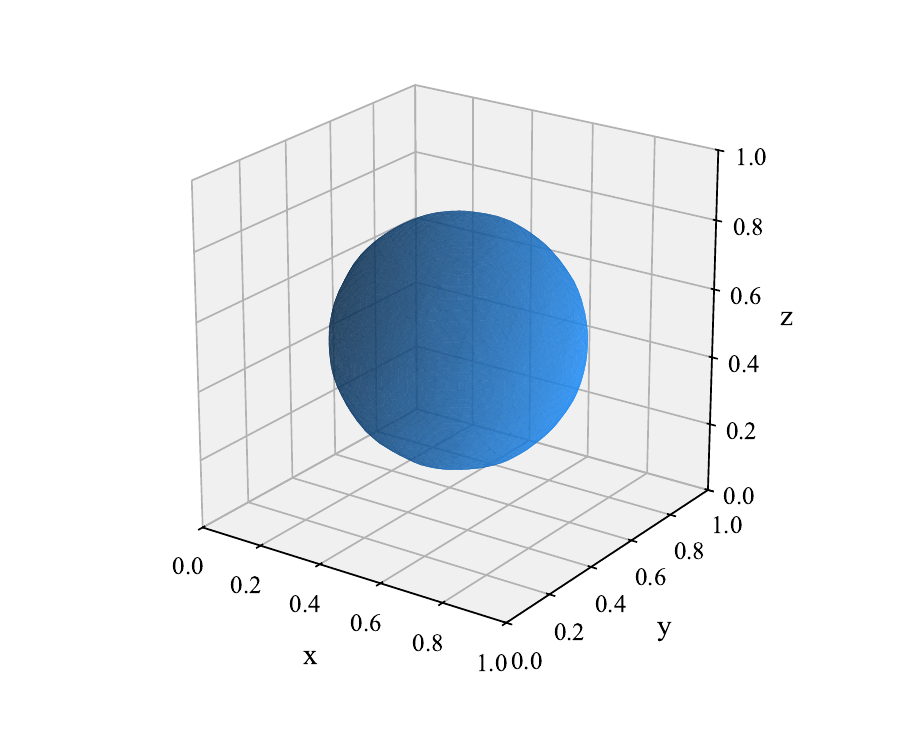}
		\caption{}
		\label{fig:pretrain_3d_phi}
	\end{subfigure}
	\caption{Initialize $\phi_{\omega}(\mathbf{x})$. (a) The output of the pre-trained $\phi_{\omega}(\mathbf{x})$ in our 2D experiment, where the red curve indicates its zero level set. (b) The output of the pre-trained $\phi_{\omega}(\mathbf{x})$ in our 3D experiment, where we plot its zero-level-set isosurface.}
	\label{fig:pretrain_phi}
\end{figure}

\begin{figure}[htbp]
	\centering
	\begin{subfigure}[b]{0.3\textwidth}
		\centering
		\includegraphics[width=\textwidth]{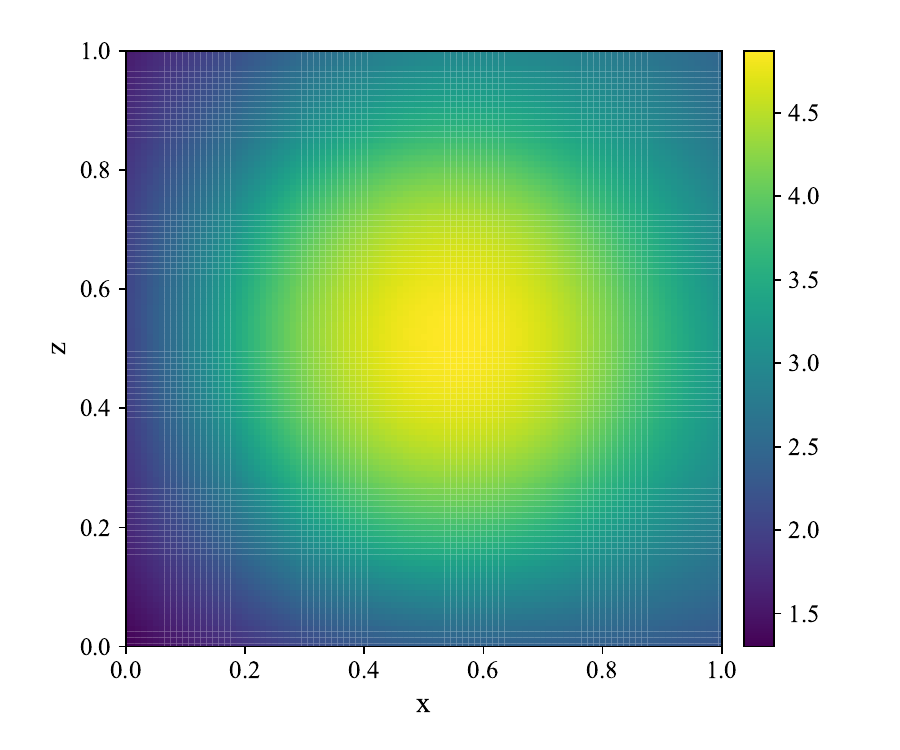}
		\caption{}
		\label{fig:pretrain_2d_u}
	\end{subfigure}
	\hspace{0.02\textwidth}
	\begin{subfigure}[b]{0.3\textwidth}
		\centering
		\includegraphics[width=\textwidth]{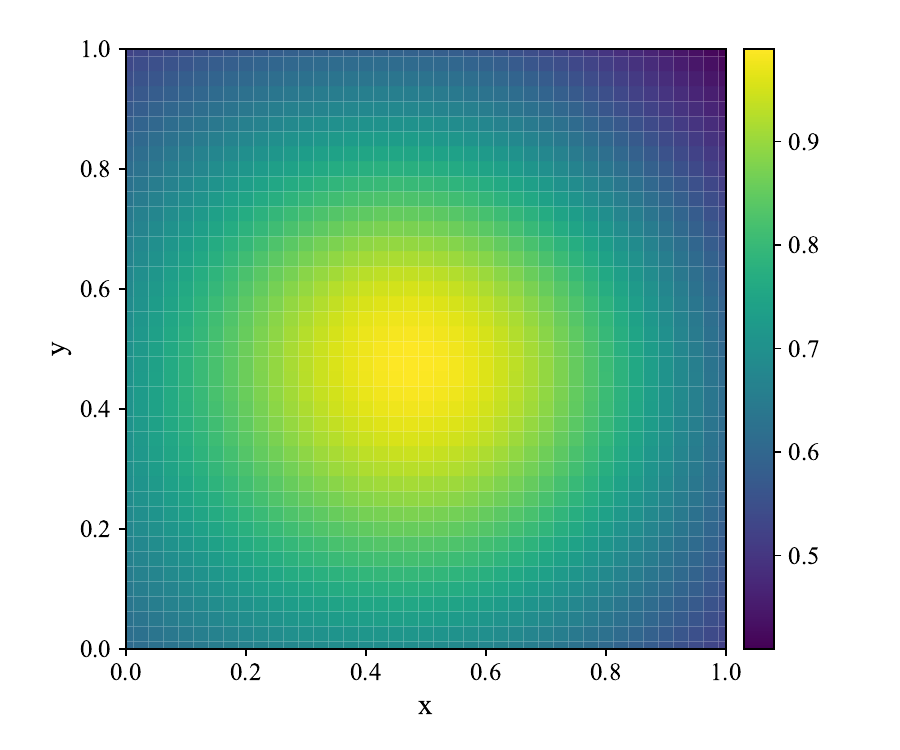}
		\caption{}
		\label{fig:pretrain_3d_u}
	\end{subfigure}
	\caption{Initialize $U_{\theta}(\mathbf{x})$. (a) The output of the pre-trained $U_{\theta}(\mathbf{x})$ in 2D experiment. (b) The output of the pre-trained $U_{\theta}(\mathbf{x})$ in 3D experiment, where we only display it at $z=0$.}
	\label{fig:pretrain_U}
\end{figure}

\subsection{2D example}


\subsubsection{Example 1: A circle and a parallelogram}

Figure \ref{fig:case1_density}\,(Left) shows the true model, where the density-contrast value $f=1\,\mathrm{g}/\mathrm{cm}^3$. The computational domain is $\Omega = (0, 1)^2$\,km, and the gravity data $\mathbf{g}=\nabla U$ are measured along $\Sigma_0=\partial\Omega$. There are $201$ measurements uniformly distributed on each of the four domain boundaries, yielding $N_b = 804$ boundary evaluations in total, where four corner points shared by adjacent boundaries are counted twice. The gravity data are simulated according to the integral equation,
\begin{equation}\label{data_simulation}
\nabla U(\mathbf{y};\rho)=\gamma\int_\Omega \nabla_{\mathbf{y}}K(\mathbf{y},\mathbf{x})\rho(\mathbf{x})\,\mathrm{d}\mathbf{x}\,,
\end{equation}
where the kernel function $K(\mathbf{y},\mathbf{x})$ is in the form of equation (\ref{eqn2}). The red line in Figure \ref{fig:case1_data} depicts the measured gravity data.

To train the level-set PINNs, $100^2$ collocation points are randomly sampled in the computational domain $\Omega$, which are fixed during the training after sampled. In addition, the adaptive refinement strategy described in Section~3.4 is employed to enhance the resolution near the evolving interface. A static pool of $50{,}000$ candidate points is uniformly sampled across the computational domain $\Omega$. During each training iteration, candidates falling within the narrow band of the interface, defined by $|\phi_\omega(\mathbf{x})| < C_{\mathrm{refine}}$ with $C_{\mathrm{refine}} = 0.18$, are dynamically selected. These points are subsequently merged with the fixed collocation points to construct the full set, $\big\{\mathbf{x}_i^p\big\}_{i=1}^{N_p}$, which is used to evaluate the loss terms $\mathcal{L}_{pde}$, $\mathcal{L}_{eik}$, and $\mathcal{L}_r$.

To enforce the far-field boundary condition, $N_{far} = 1{,}000$ points are sampled with uniformly spaced angles and random radii $r \in (1{,}000, 2{,}000)$ from the center of $\Omega$. The loss term $\mathcal{L}_{far}$ enforces $\nabla U_\theta = 0$ at these locations.

We employ the typical Adam algorithm to minimize the total loss, where the selection of weights $\lambda_i$ ($i=1,\cdots,5$) is shown in Table \ref{tab:case1_params}. As discussed in Section \ref{subsec_loss}, determining optimal weights in PINNs can be challenging, which is not our goal, and we follow a heuristic approach to select their values. The principle is to ensure that each loss term is appropriately minimized during training iterations. The relatively large values of $\lambda_1$ and $\lambda_2$ reflect the need to penalize the boundary and far-field loss terms rigorously; the regularization weights $\lambda_4$ and $\lambda_5$ are taken small, consistent with the auxiliary role of $\mathcal{L}_r$ and $\mathcal{L}_\Theta$. Figure~\ref{fig:case1_loss} plots the value of each loss term over $25{,}000$ training iterations. It shows that all loss terms decrease steadily and remain reasonable in magnitude throughout training, indicating that the selected weights yield a well-balanced optimization. In this example, we take $\tau=10^{-2}$ and $\tilde{\tau}=2\times10^{-2}$ for the interface-aware backpropagation. Table~\ref{tab:case1_params} summarizes the settings of hyperparameters.

\begin{table}[htbp]
\centering
\caption{Hyperparameter settings for numerical examples 1 to 4}
\label{tab:case1_params}
\begin{tabular}{ccccccccc}
\toprule
Parameter & $\lambda_1$ & $\lambda_2$ & $\lambda_3$ & $\lambda_4$ & $\lambda_5$ & $\tau$ & $\tilde{\tau}$ & $C_{\mathrm{refine}}$ \\
\midrule
Value & $20$ & $20$ & $1$ & $5\times10^{-3}$ & $10^{-5}$ & $10^{-2}$ & $2\times10^{-2}$ & $0.18$ \\
\bottomrule
\end{tabular}
\end{table}

Figure \ref{fig:case1_density}\,(Right) shows the inverted density $\rho_{\omega}(\mathbf{x})$; Figure~\ref{fig:case1_data} plots the predicted gravity data $\mathbf{g}=\nabla U_\theta$ as a dashed blue line. 
The predicted data perfectly matches the measured data, which is consistent with the vanishing behavior of $\mathcal{L}_{bc}$ shown in Figure \ref{fig:case1_loss}. The inverted solution successfully recovers the shape of the density model. Given that the gravity inverse problem is ill-posed, the solution is sufficiently good.

\begin{figure}[htbp]
	\centering
	\includegraphics[width=0.58\textwidth]{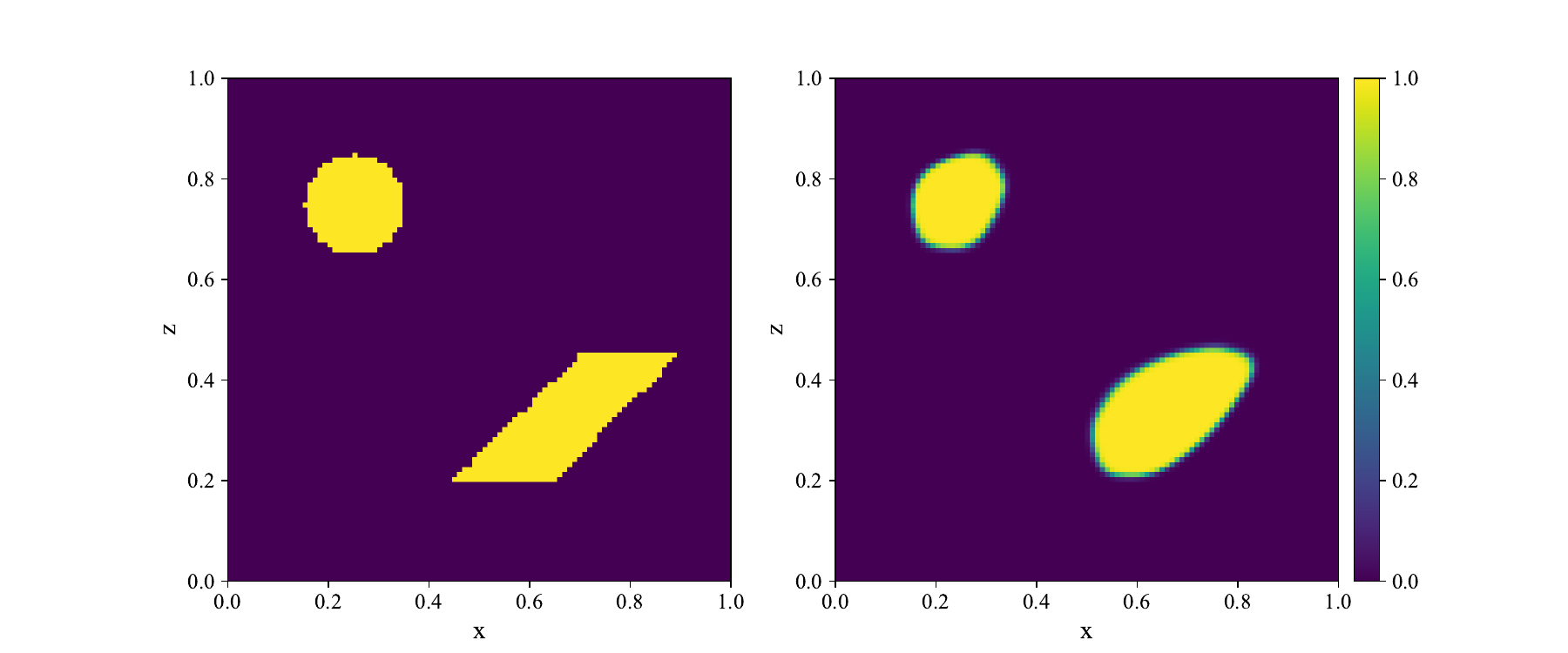}
	\caption{Example~1: A circle and a parallelogram. Left: True density. Right: Inverted density.}
	\label{fig:case1_density}
\end{figure}

\begin{figure}[htbp]
	\centering
	\includegraphics[width=0.88\textwidth]{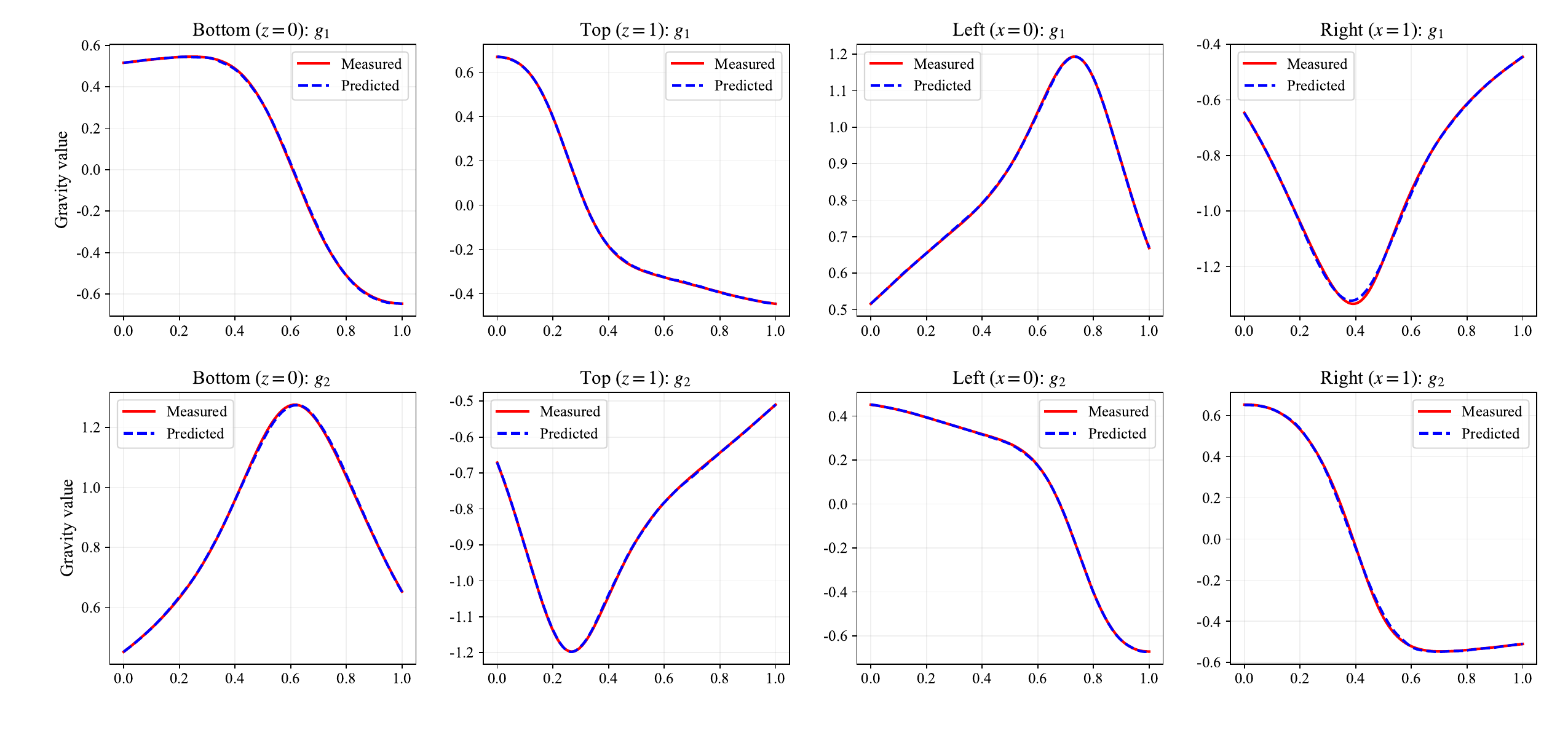}
	\caption{Example~1: Gravity-data fitting. The first row shows the fitting for $g_1:=\frac{\partial U}{\partial x}$ along the four boundaries, and the second row shows the fitting for $g_2:=\frac{\partial U}{\partial z}$. The red line plots the measured data, and the dashed blue line plots the predicted data. The unit of gravity data is milligal.}
	\label{fig:case1_data}
\end{figure}

\begin{figure}[htbp]
	\centering
	\includegraphics[width=0.50\textwidth]{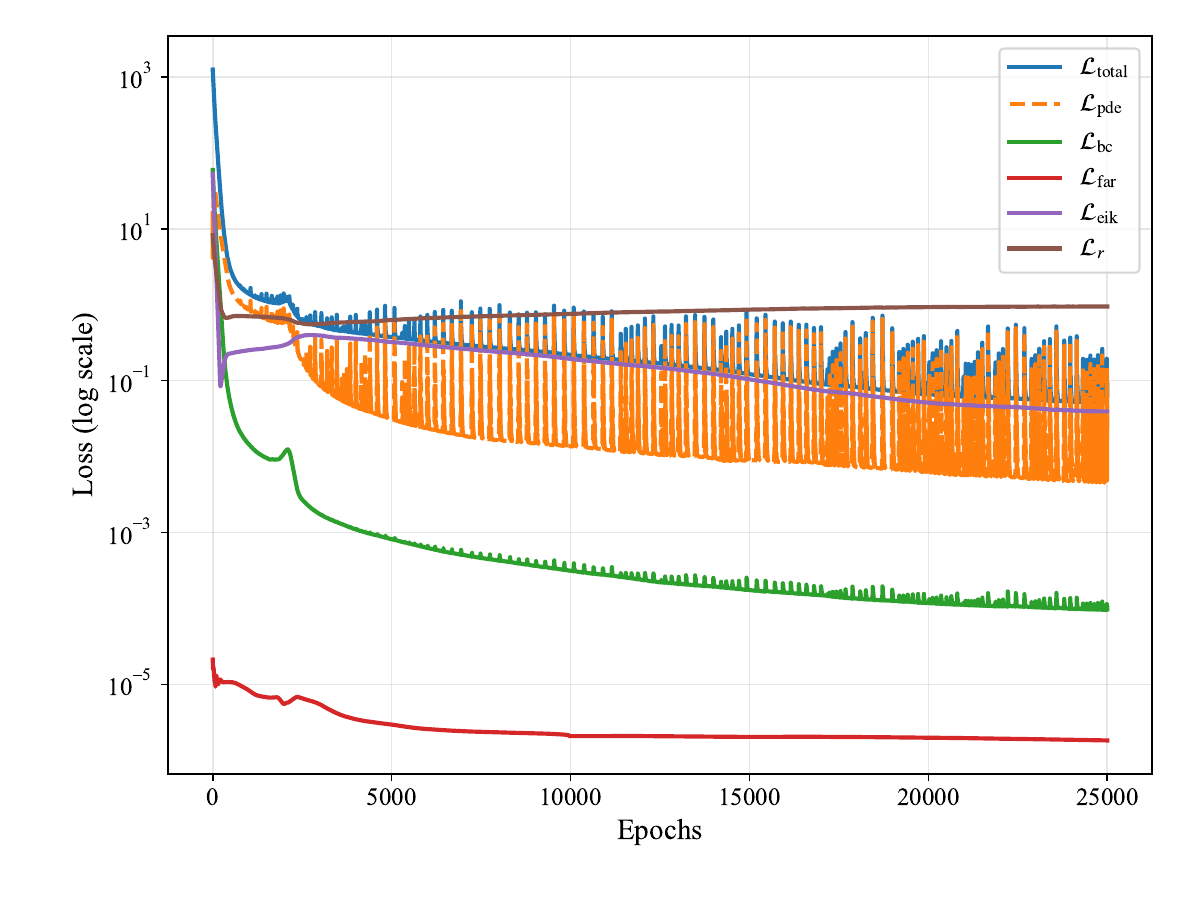}
	\caption{Example~1: Convergence history. The values of the individual loss terms guide the selection of their weighting coefficients.}
	\label{fig:case1_loss}
\end{figure}

As a baseline comparison, we consider a standard PINNs approach where the density function $\rho_{\omega}(\mathbf{x})$ is directly represented by a ReLU neural network. Accordingly, in the loss function (\ref{eqn3.7}), the terms involving the level set function, $\mathcal{L}_{eik}$ and $\mathcal{L}_r$, are removed, while all other settings remain unchanged. Figures \ref{fig:case1_comparison1_density}, \ref{fig:case1_comparison1_data}, and \ref{fig:case1_comparison1_loss} present the corresponding results. Although both the data fitting and the convergence plots appear excellent, the inverted density function deviates substantially from the true solution. In this case, the recovered density distribution contains almost no meaningful information.

Furthermore, we modify the above standard PINNs by incorporating a sigmoid activation function in the output layer of the density network to constrain the values of $\rho_\omega(\mathbf{x})$ within the range $[0, 1]$; all other settings are kept unchanged. The results are presented in Figures \ref{fig:case1_comparison2_density}, \ref{fig:case1_comparison2_data}, and \ref{fig:case1_comparison2_loss}. While this modification leads to an improvement in the predicted density distribution, the results remain substantially inferior to those obtained by our level-set PINNs.

\begin{figure}[htbp]
	\centering
	\includegraphics[width=0.58\textwidth]{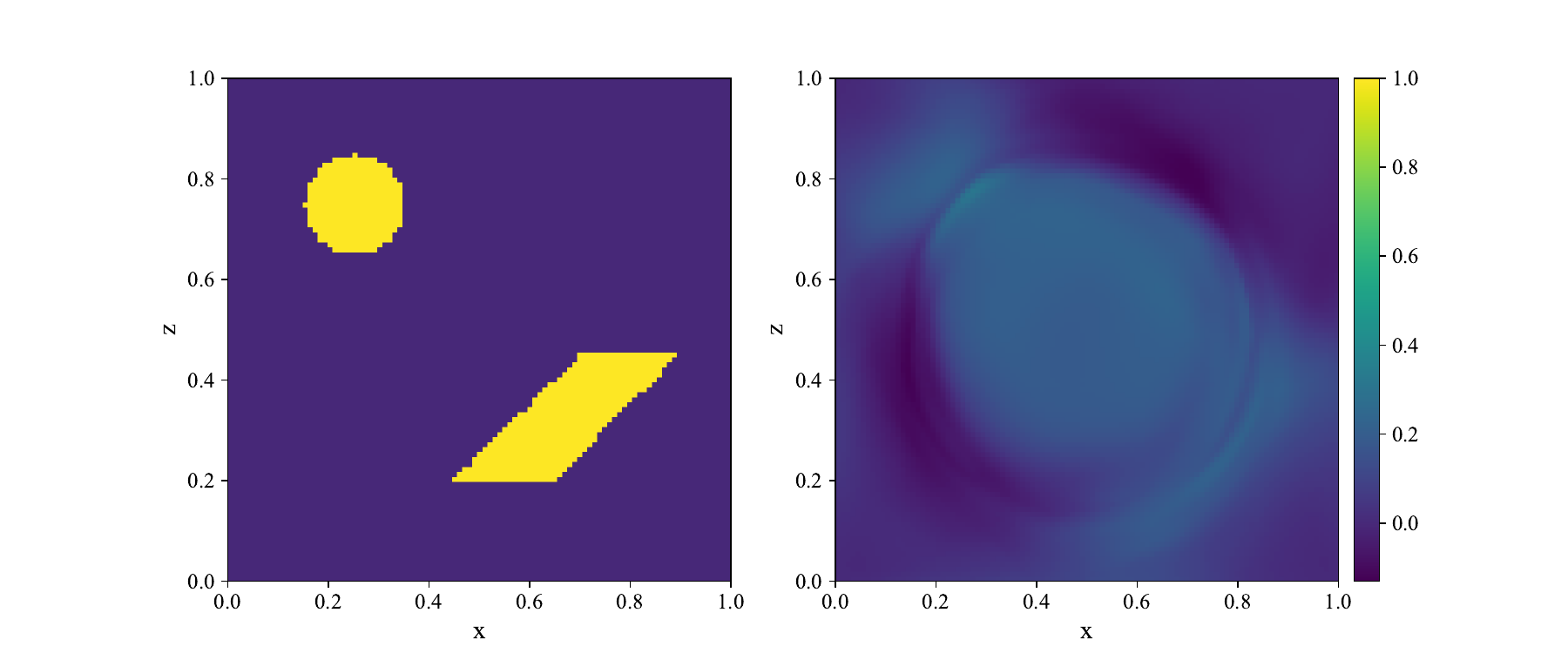}
	\caption{Baseline Comparison for Example 1: Standard PINNs. Left: True density. Right: Inverted density.}
	\label{fig:case1_comparison1_density}
\end{figure}

\begin{figure}[htbp]
	\centering
	\includegraphics[width=0.88\textwidth]{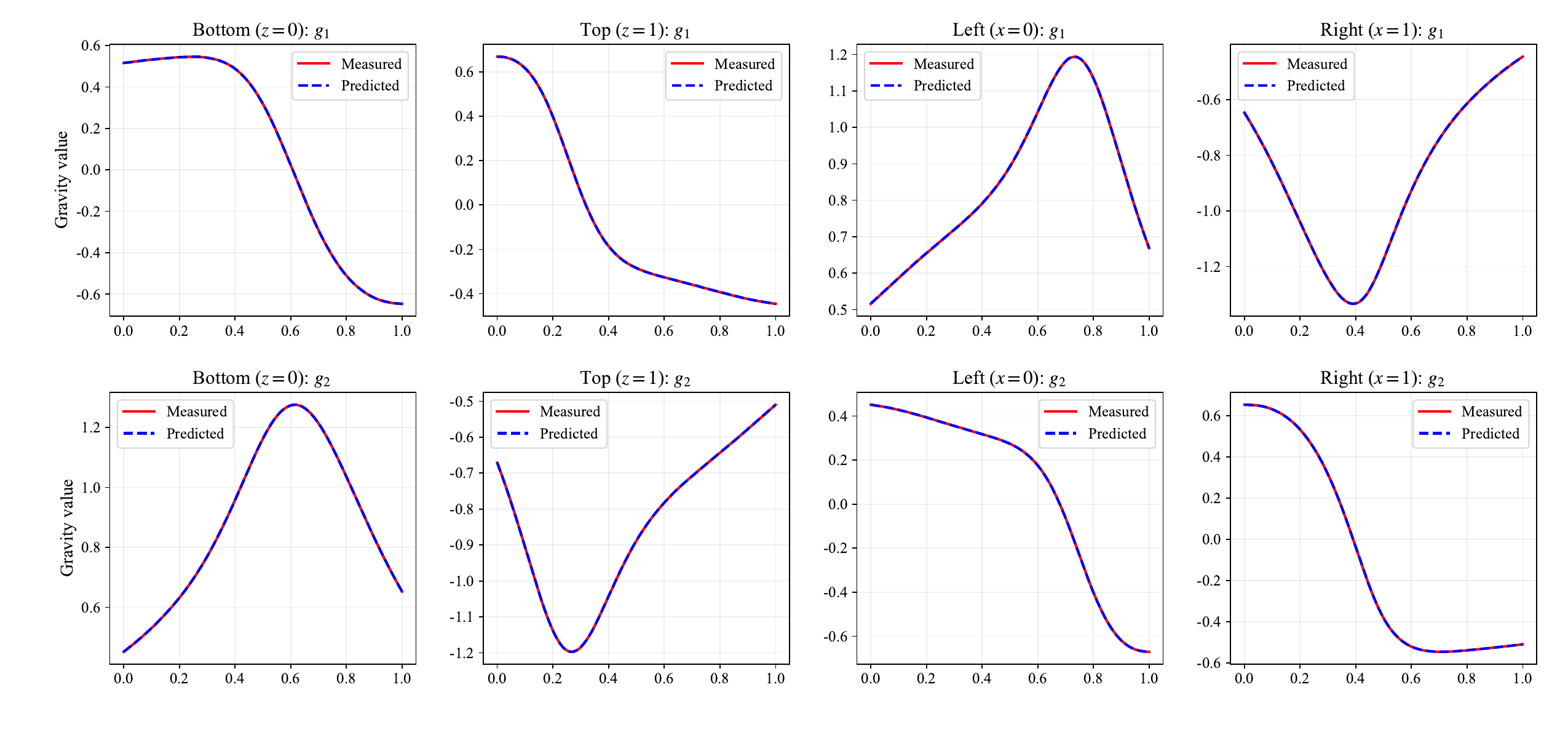}
	\caption{Baseline Comparison for Example 1: Standard PINNs. The performance of data fitting. The first row shows the fitting for $g_1$ along the four boundaries, and the second row shows the fitting for $g_2$.}
	\label{fig:case1_comparison1_data}
\end{figure}

\begin{figure}[htbp]
	\centering
	\includegraphics[width=0.50\textwidth]{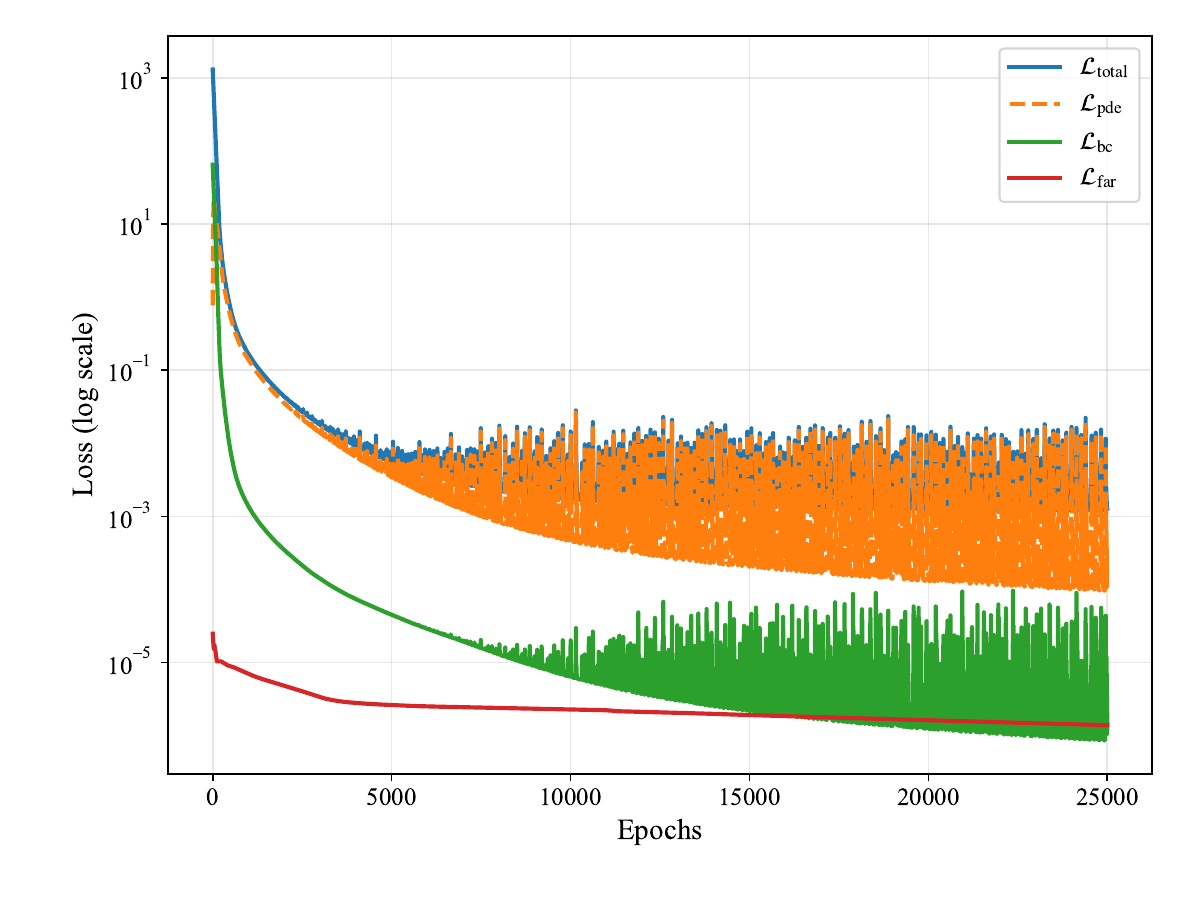}
	\caption{Baseline Comparison for Example 1: Standard PINNs. Convergence history.}
	\label{fig:case1_comparison1_loss}
\end{figure}

\begin{figure}[htbp]
	\centering
	\includegraphics[width=0.58\textwidth]{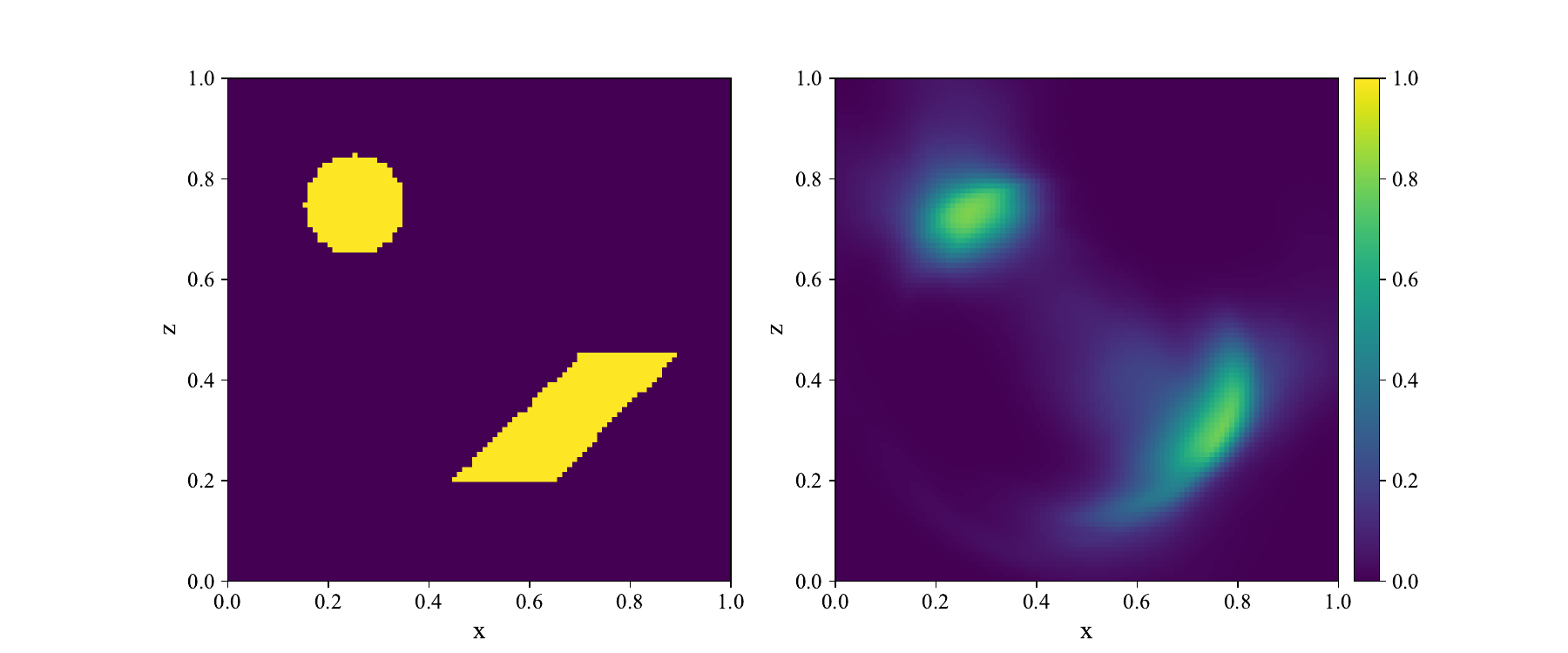}
	\caption{Comparison for Example 1: Standard PINNs with a sigmoid activation for $\rho_\omega$. Density functions. Left: True density. Right: Inverted density.}
	\label{fig:case1_comparison2_density}
\end{figure}

\begin{figure}[htbp]
	\centering
	\includegraphics[width=0.88\textwidth]{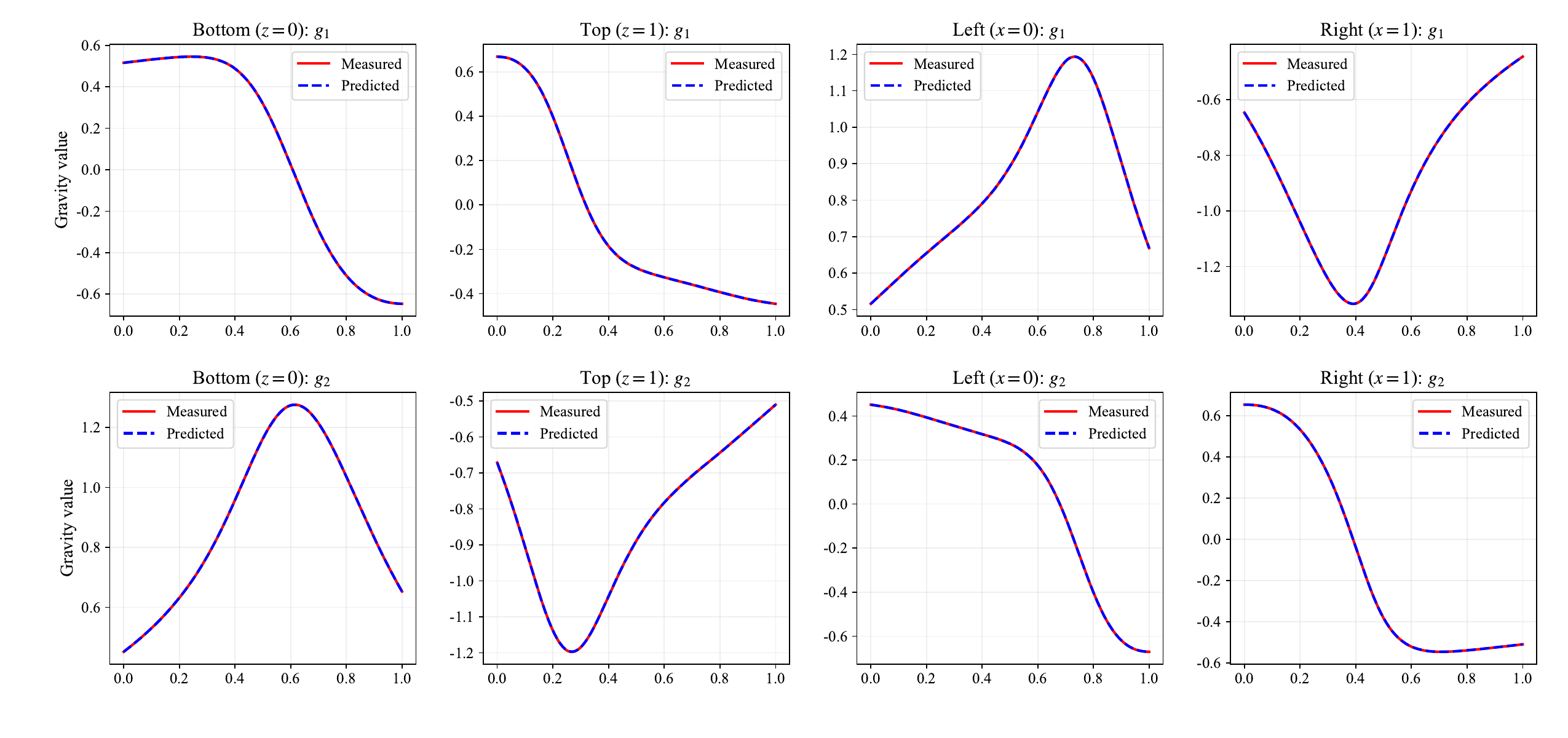}
	\caption{Comparison for Example 1: Standard PINNs with a sigmoid activation for $\rho_\omega$. The performance of data fitting. The first row shows the fitting for $g_1$ along the four boundaries, and the second row shows the fitting for $g_2$.}
	\label{fig:case1_comparison2_data}
\end{figure}

\begin{figure}[htbp]
	\centering
	\includegraphics[width=0.50\textwidth]{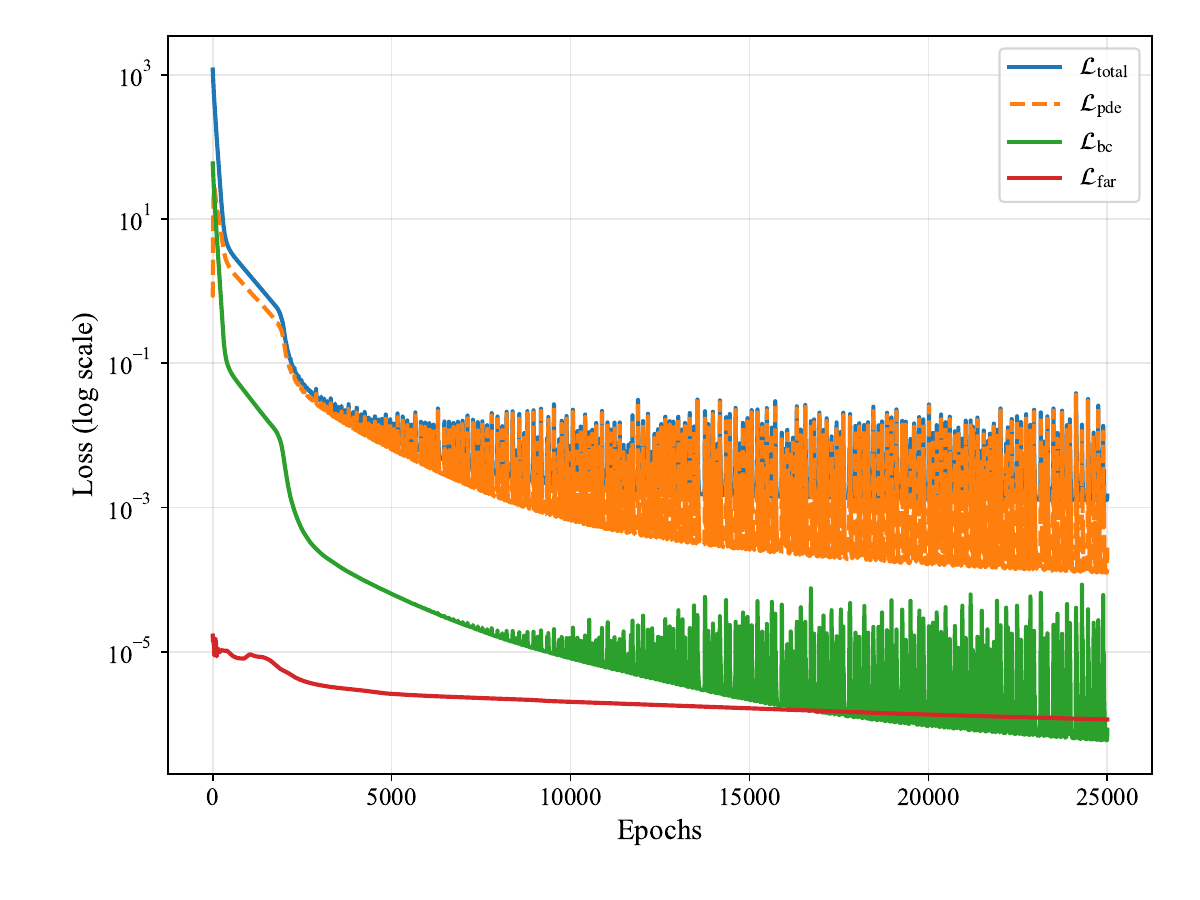}
	\caption{Comparison for Example 1: Standard PINNs with a sigmoid activation for $\rho_\omega$. Convergence history.}
	\label{fig:case1_comparison2_loss}
\end{figure}

\subsection{3D examples}
\subsubsection{Example 2: Four spherical anomalies}

Figure \ref{fig:case5_density_slices}\,(a) (Left) shows the true model, which contains four spherical anomalies with a constant density contrast of $f=1\mathrm{g}/\mathrm{cm}^3$. The 3D computational domain is $\Omega = (0, 1)^3$\,km, and the gravity data $\mathbf{g}=\nabla U$ are measured along $\partial\Omega$. There are $51^2$ measurements uniformly distributed on each of the six boundary faces, yielding a total of $N_b = 6 \times 51^2 = 15{,}606$ boundary evaluations; points along shared edges and corners are counted multiple times. The gravity data are simulated according to Equation~(\ref{data_simulation}).

To train the level-set PINNs, $40^3$ collocation points are randomly sampled in the computational domain $\Omega$ and remain fixed throughout training. The adaptive refinement strategy of Section~3.4 is applied in the same manner as in Example~1. For the far-field boundary condition, $N_{far} = 5{,}000$ points are sampled in the exterior region at distances between $1{,}000$ and $2{,}000$ from the center of $\Omega$, with uniformly spaced polar and azimuthal angles and random radial distances. All hyperparameters, including the loss weights and the interface-aware backpropagation parameters, are set to the same values as in Example~1 (see Table~\ref{tab:case1_params}).

Figure~\ref{fig:case5_density_slices}\,(a) (Right) shows the inverted density model, and Figure~\ref{fig:case5_density_slices}\,(b) plots the cross-sectional slices at $y=0.2$\,km, $y=0.8$\,km, and $z=0.5$\,km, respectively. The inverted density distribution maintains sharp interfaces that successfully recover the true anomaly boundaries. Figure~\ref{fig:case5_data} illustrates the gravity-data fitting on the top  measurement surface ($z=1$\,km), confirming a good match between the predicted and measured data.

\begin{figure}[htbp]
	\centering
	\begin{subfigure}[b]{0.70\textwidth}
		\centering
		\includegraphics[width=\textwidth]{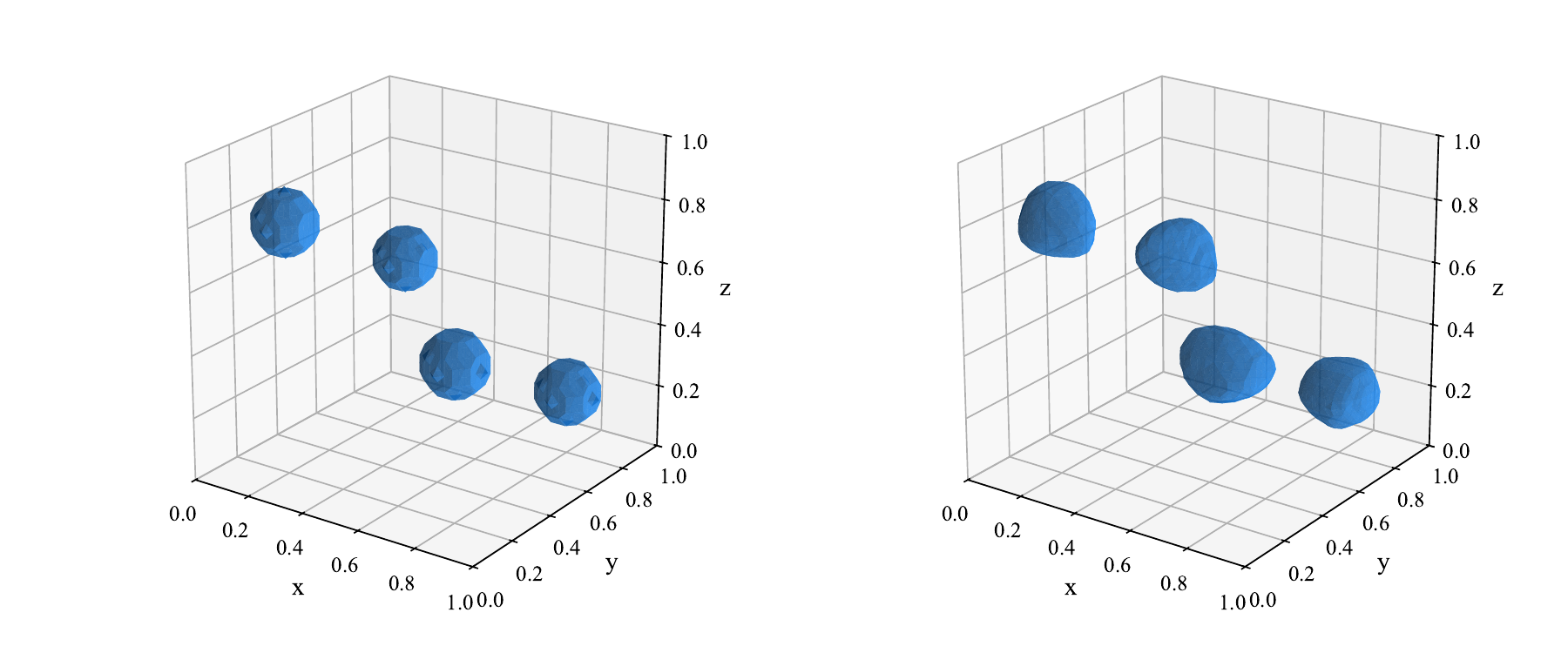}
		\caption{Left: True density. Right: Inverted density.}
	\end{subfigure}\\[6pt]
	\begin{subfigure}[b]{0.82\textwidth}
		\centering
		\includegraphics[width=\textwidth]{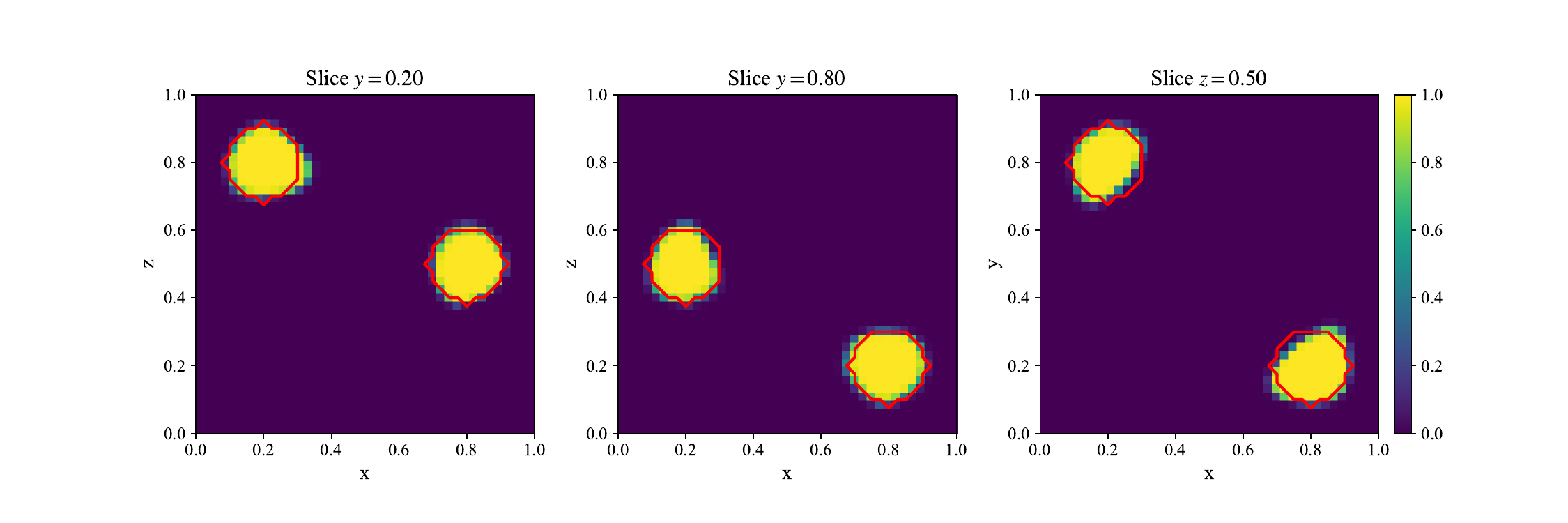}
		\caption{Slices of the inverted model, where the true model is outlined in red.}
	\end{subfigure}
	\caption{Example~2: Four spherical anomalies. (a) True and inverted density models. (b) Slices of the inverted model.}
	\label{fig:case5_density_slices}
\end{figure}

\begin{figure}[htbp]
	\centering
	\includegraphics[width=0.70\textwidth]{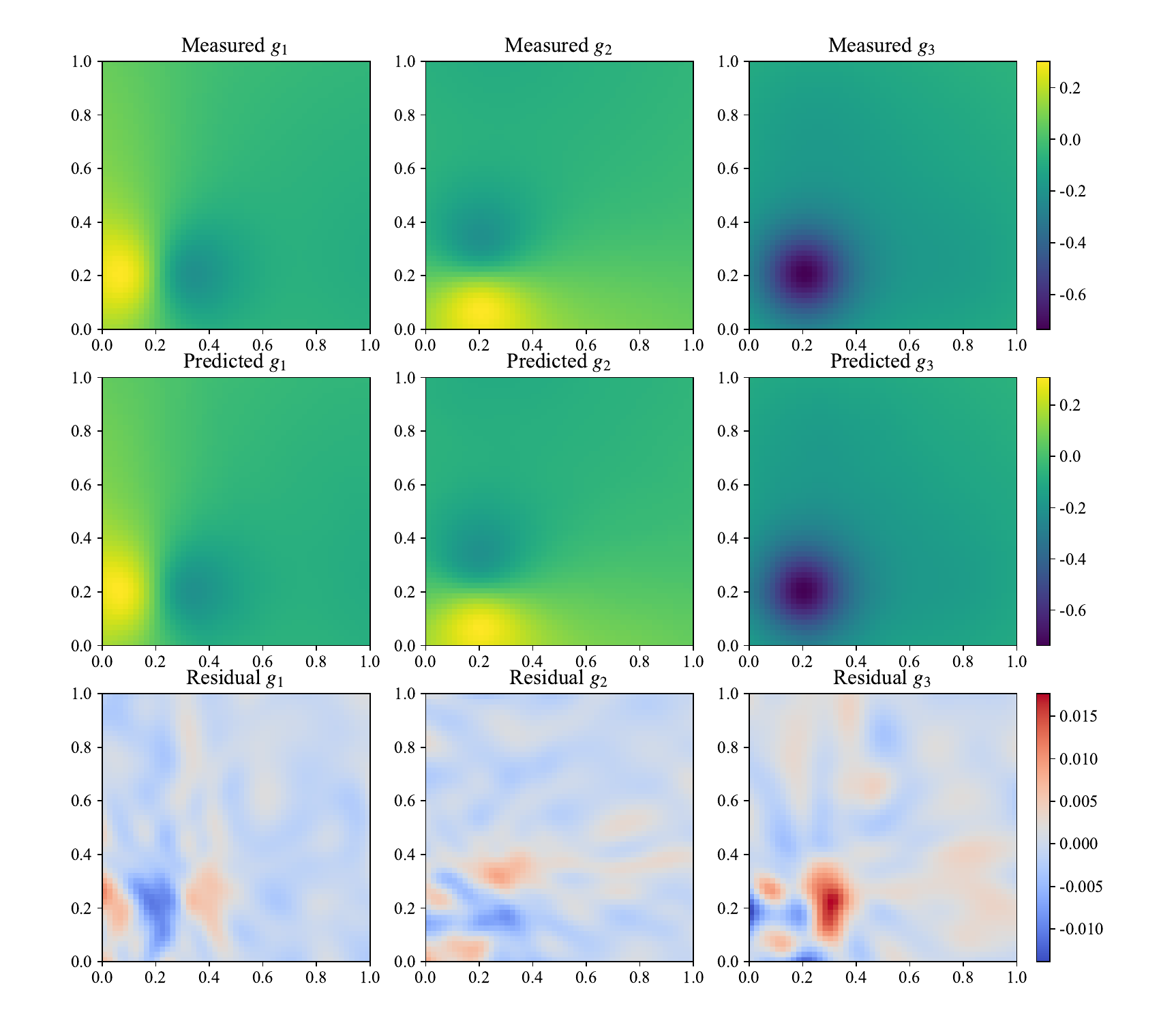}
	\caption{Example~2: Gravity-data fitting on the top measurement surface ($z=1$\,km). The first, second, and third rows display the measured data, predicted data, and residuals, respectively; $(g_1, g_2, g_3) = \nabla U$.}
	\label{fig:case5_data}
\end{figure}


\subsubsection{Example 3: Anomalies with mixed geometries}

Figure \ref{fig:case6_density_slices}(a) (Left) shows the true model, which features three anomalies with distinct geometries and a constant density contrast of $f = 1\mathrm{g/cm^3}$. The computational setup, including the inversion domain, measurement configuration, collocation point sampling, adaptive refinement, and far-field point distribution, is identical to that of Example~2. All hyperparameters are likewise maintained as specified in Table~\ref{tab:case1_params}.

Figure~\ref{fig:case6_density_slices} presents the inverted density model, where Figure \ref{fig:case6_density_slices}(a) (Right) displays the zero-level-set isosurface, and Figure~\ref{fig:case6_density_slices}\,(b) plots the cross-sectional slices at $x=0.2$\,km, $x=0.5$\,km, and $x=0.8$\,km. The inverted model captures the overall geometries of the three anomalies; in particular, the dipping parallelepiped structure is successfully recovered. Considering the ill-posedness and difficulty of gravity inversion, this reconstruction is adequate and satisfactory.
Figure~\ref{fig:case6_data} illustrates the gravity-data fitting on the top measurement surface ($z=1$\,km), confirming a good agreement between the predicted and measured data.

\begin{figure}[htbp]
	\centering
	\begin{subfigure}[b]{0.70\textwidth}
		\centering
		\includegraphics[width=\textwidth]{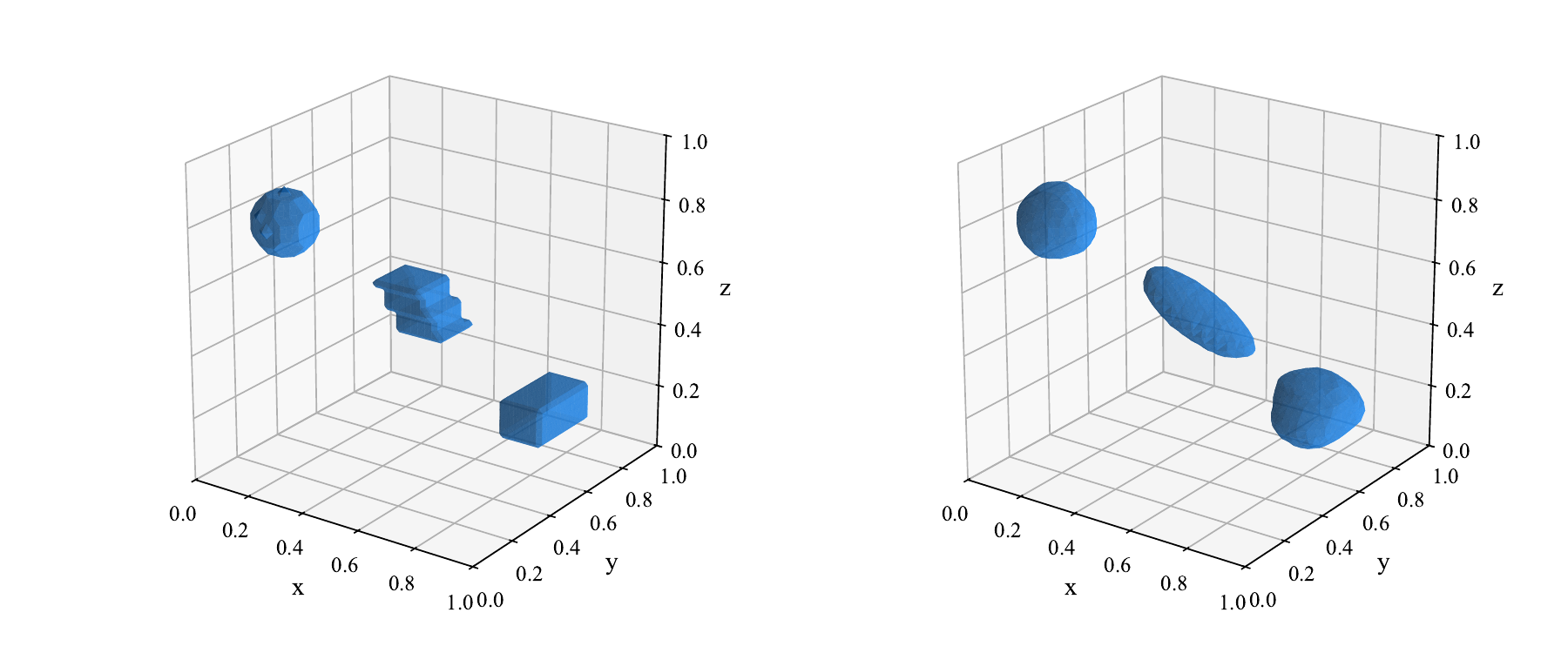}
		\caption{Left: True density. Right: Inverted density.}
	\end{subfigure}\\[-5pt]
	\begin{subfigure}[b]{0.82\textwidth}
		\centering
		\includegraphics[width=\textwidth]{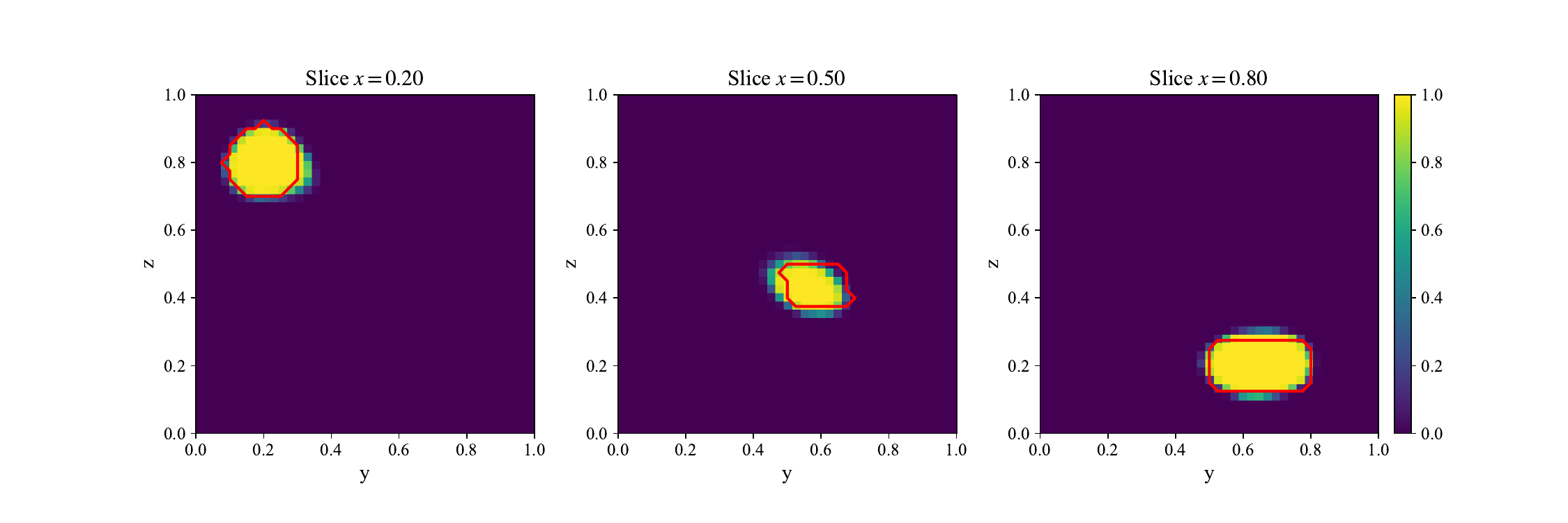}
		\caption{Slices of the inverted model, where the true model is outlined in red.}
	\end{subfigure}
	\caption{Example~3: Anomalies with mixed geometries. (a) True and inverted density models. (b) Slices of the inverted model.}
	\label{fig:case6_density_slices}
\end{figure}

\begin{figure}[htbp]
	\centering
	\includegraphics[width=0.70\textwidth]{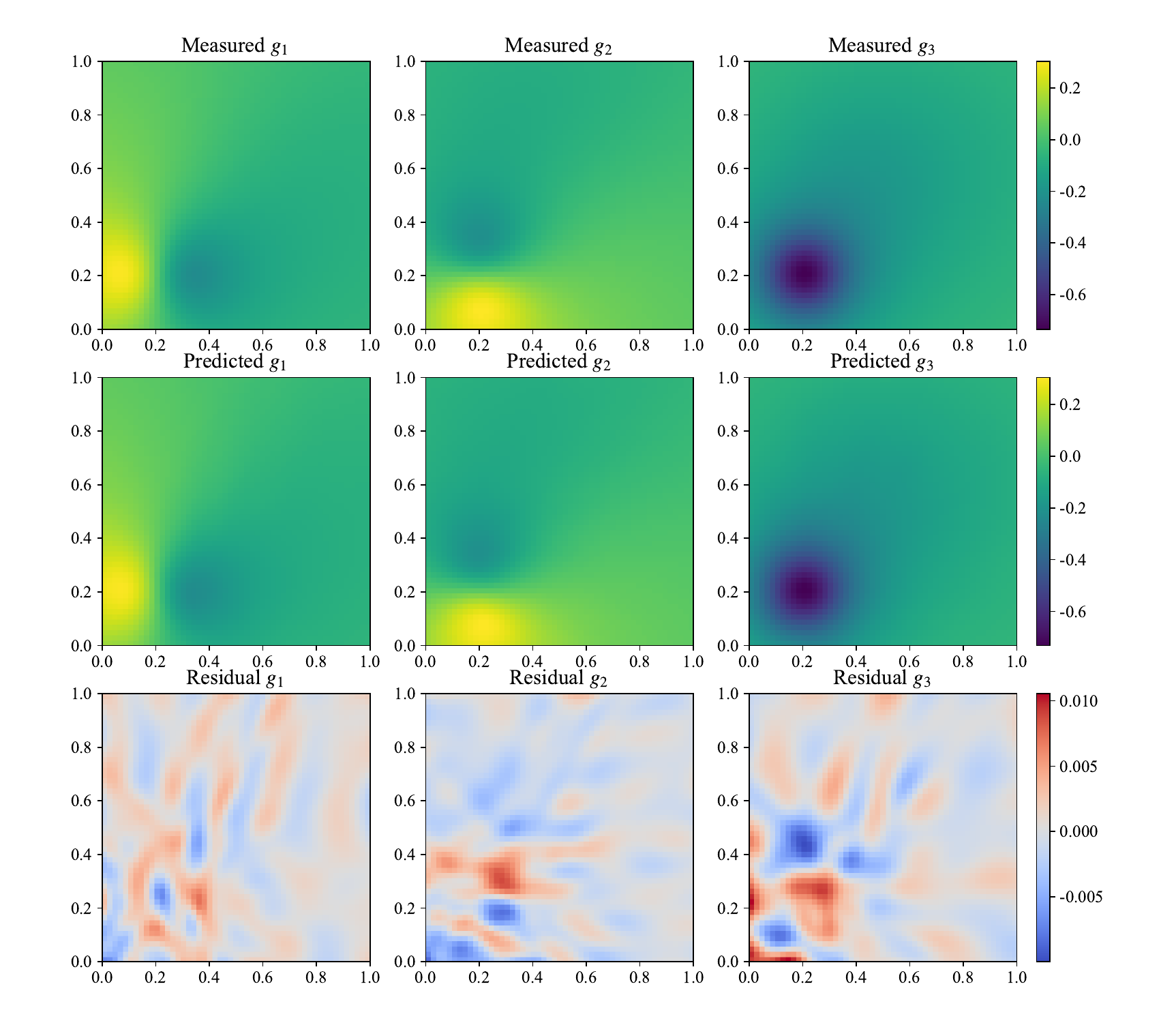}
	\caption{Example~3: Gravity-data fitting on the top measurement surface ($z=1$\,km). The first, second, and third rows display the measured data, predicted data, and residuals, respectively; $(g_1, g_2, g_3) = \nabla U$.}
	\label{fig:case6_data}
\end{figure}


\subsubsection{Example 4: Anomalies with depth-dependent density contrast}

In the last example, we consider a depth-dependent density contrast for the mass anomalies, defined as $f(\mathbf{x}) = 0.5 + z$. The geometry of the true model is shown in Figure \ref{fig:case7_density_slices}(a) (Left). The computational setup and all hyperparameters are identical to those of Examples~2 and~3 (see Table~\ref{tab:case1_params}).

Figure~\ref{fig:case7_density_slices} presents the inverted density model, including its zero-level-set isosurface and cross-sectional slices. Figure~\ref{fig:case7_data} illustrates the gravity-data fitting on the top measurement surface ($z=1$\,km). The level-set PINNs successfully recovers the mass anomalies with varying density contrast, and the predicted data exhibits a good agreement with the measured data.

\begin{figure}[htbp]
	\centering
	\begin{subfigure}[b]{0.70\textwidth}
		\centering
		\includegraphics[width=\textwidth]{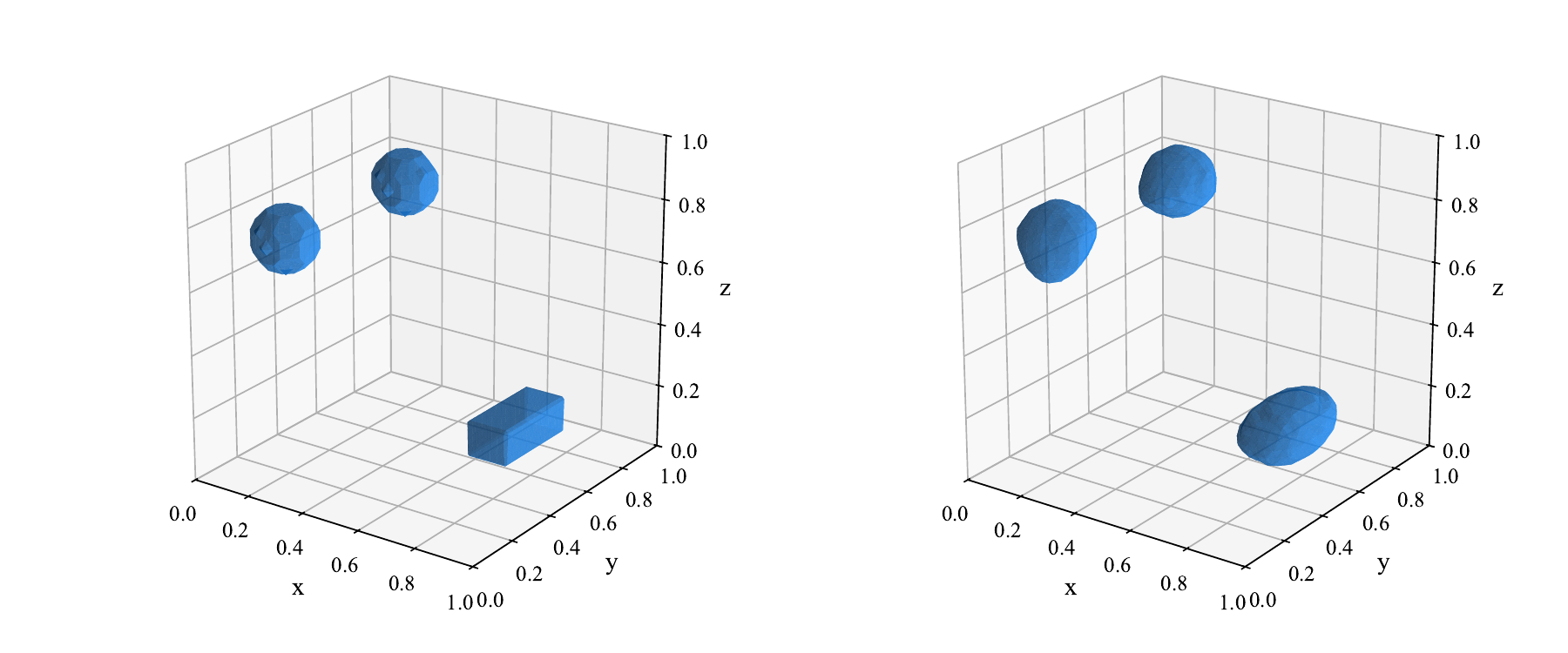}
		\caption{Left: True density. Right: Inverted density.}
	\end{subfigure}\\ [-5pt]
	\begin{subfigure}[b]{0.82\textwidth}
		\centering
		\includegraphics[width=\textwidth]{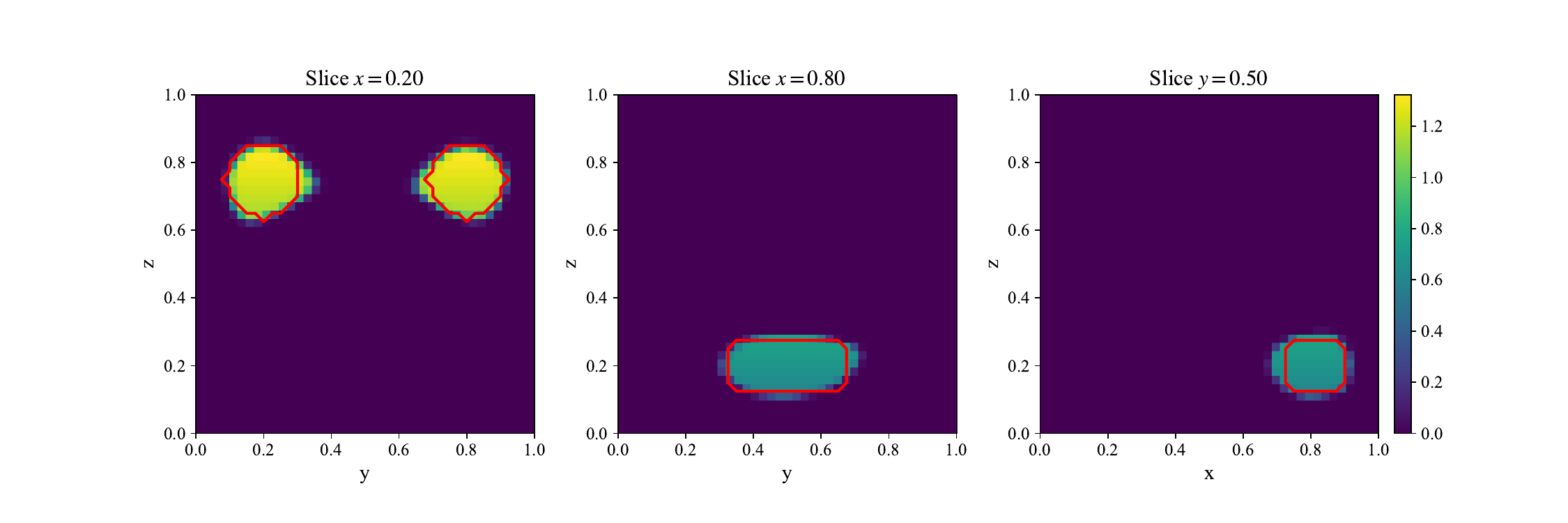}
		\caption{Slices of the inverted model, where the true model is outlined in red.}
	\end{subfigure}
	\caption{Example~4: Anomalies with depth-dependent density contrast. (a) True and inverted density models. (b) Slices of the inverted model.}
	\label{fig:case7_density_slices}
\end{figure}

\begin{figure}[htbp]
	\centering
	\includegraphics[width=0.70\textwidth]{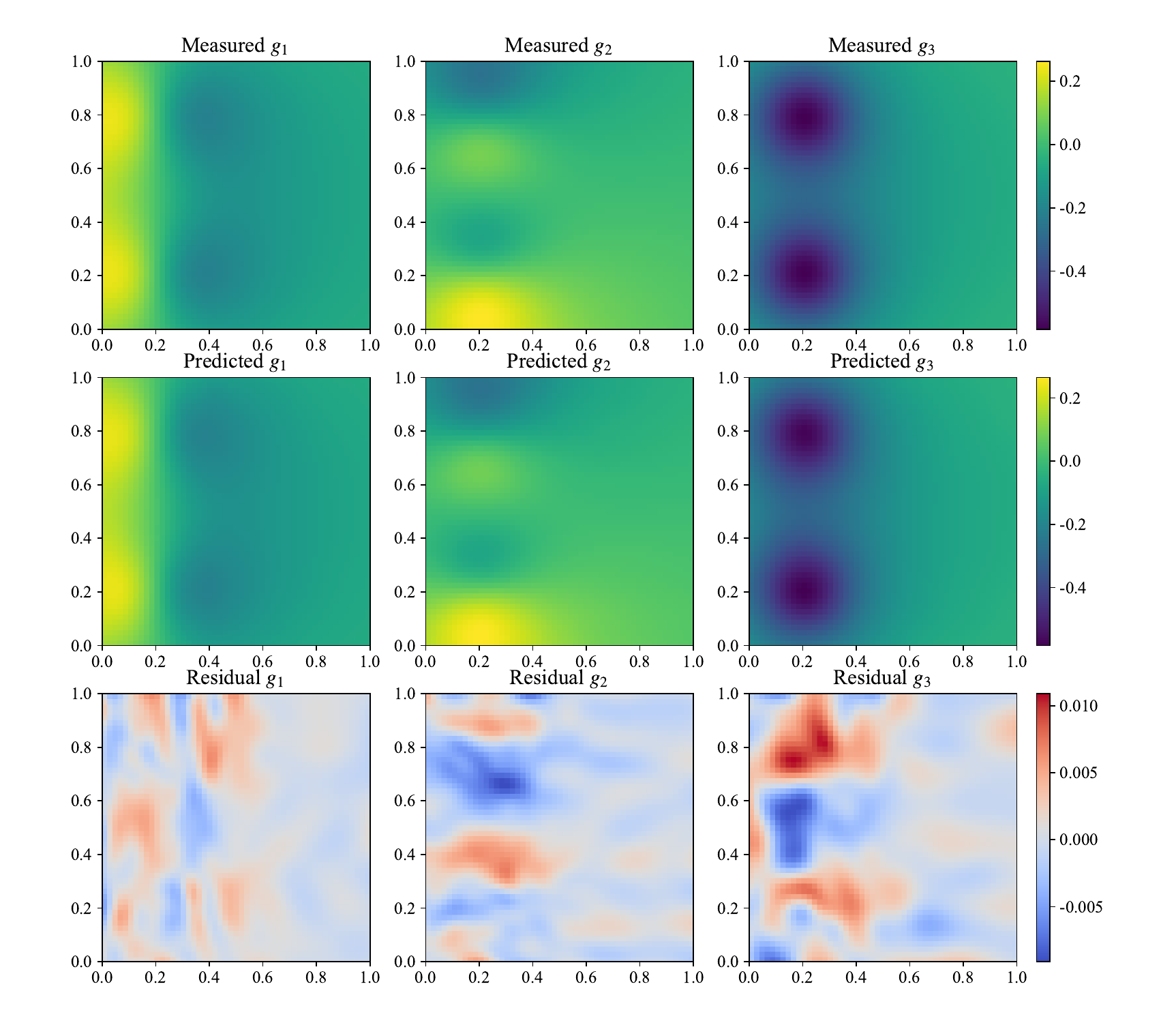}
	\caption{Example~4: Gravity-data fitting on the top measurement surface ($z=1$\,km). The first, second, and third rows display the measured data, predicted data, and residuals, respectively; $(g_1, g_2, g_3) = \nabla U$.}
	\label{fig:case7_data}
\end{figure}


\section{Conclusions} \label{sec5}
We have proposed the method of level-set physics-informed neural networks (PINNs) for solving domain inverse problems of gravimetry. 
The paradigm of PINNs is adopted to construct the physical loss function for inverse gravimetry, where the challenge is resolving unknown domains involving discontinuous interface structures. By representing the domain shape via the zero level-set of a continuous neural network, we circumvent the difficulties of directly parameterizing discontinuous interfaces under PINNs. To drive the network training, we exploit the interface evolution mechanism inherent to the level-set method. A useful contribution is the interface-aware backpropagation strategy, which redefines derivative computations near interfaces to mitigate exploding and vanishing gradients. It enables a broader support region for interface derivatives while preserving the sharp interface structure. A detailed analysis justifies its efficacy. Furthermore, we enhance the computational framework by introducing an adaptive refinement procedure for interfacial collocation points, and investigate network architecture choices based on their solution spaces and approximation properties. Extensive 2D and 3D numerical experiments confirm that level-set PINNs successfully solve domain inverse problems of gravimetry.
Beyond the scope of inverse gravimetry, our work provides broader insights into the development of neural network methods for general domain and interface reconstruction problems. Future research will look forward expanding these concepts into model-guided architectures that embed physical laws directly into unrolled neural network layers (e.g., \cite{tsameletc25}).

\section*{Acknowledgments}
Wenbin Li is supported by the Natural Science Foundation of Shenzhen (JCYJ20240813104841055), and the National Science and Technology Major Project for Deep Earth Exploration and Mineral Prospecting (2024ZD1002700).

\bibliographystyle{siamplain}
\bibliography{myref}

\end{document}